\documentclass[a4paper, 11pt]{amsart}
\usepackage[utf8]{inputenc}
\usepackage[english]{babel}
\usepackage[T1]{fontenc}
\usepackage{amsmath, amssymb, amsfonts}
\usepackage{enumerate}
\usepackage{graphicx}
\usepackage{comment}
\usepackage[hmargin=1.5cm, vmargin=2.5cm]{geometry}
\usepackage[colorlinks=true, citecolor=blue]{hyperref}

\usepackage[svgnames]{xcolor} 

\title{Cutoff profiles for quantum Lévy processes and quantum random transpositions}
\author{Amaury Freslon}
\author{Lucas Teyssier}
\author{Simeng Wang}
\email{amaury.freslon@math.u-psud.fr}
\email{lucas.teyssier@univie.ac.at}
\email{simeng.wang@math.u-psud.fr}
\address{A. Freslon, Universit\'e Paris-Saclay, CNRS, Laboratoire de Math\'ematiques d'Orsay, 91405 Orsay, France}
\address{L. Teyssier, Universität Wien, Fakultät für Mathematik, 1090 Vienna, Austria}
\address{S. Wang, Universit\'e Paris-Saclay, CNRS, Laboratoire de Math\'ematiques d'Orsay, 91405 Orsay, France}
\date{}

\theoremstyle{plain}

\newtheorem{thm}{Theorem}[section]
\newtheorem{lem}[thm]{Lemma}
\newtheorem{prop}[thm]{Proposition}
\newtheorem{cor}[thm]{Corollary}

\theoremstyle{definition}
\newtheorem{de}[thm]{Definition}

\newtheorem{rem}[thm]{Remark}

\DeclareMathOperator{\FMeix}{Meix^{+}}
\DeclareMathOperator{\FPoiss}{Poiss^{+}}
\DeclareMathOperator{\id}{id}

\DeclareMathOperator{\Span}{Span}
\DeclareMathOperator{\var}{var}

\newcommand{\B}{\mathcal{B}}
\newcommand{\C}{\mathbf{C}}

\newcommand{\D}{\Delta}

\newcommand{\E}{\mathbb{E}}
\newcommand{\G}{\mathbb{G}}
\newcommand{\N}{\mathbf{N}}
\newcommand{\R}{\mathbf{R}}

\renewcommand{\O}{\mathcal{O}}

\newcommand{\summ}[2]{\sum_{#1}^{#2}}
\DeclareMathOperator{\Poiss}{Poiss}

\DeclareMathOperator{\dvt}{d_{TV}}
\newcommand{\abs}[1]{\left\lvert #1 \right\rvert}

\newcommand{\cg}{\left[}
\newcommand{\cd}{\right]}
\newcommand{\ag}{\left\{}
\newcommand{\ad}{\right\}}
\newcommand{\pg}{\left(} 
\newcommand{\pd}{\right)}

\newcommand{\lucas}[1]{\textcolor{orange}{#1}}

\begin{document}

\maketitle

\begin{center}
    \textbf{Abstract}
\end{center}

We consider a natural analogue of Brownian motion on free orthogonal quantum groups and prove that it exhibits a cutoff at time $N\ln(N)$.
Then, we study the induced classical process on the real line and compute its atoms and density. This enables us to find the cutoff profile, which involves free Poisson distributions and the semi-circle law. We prove similar results for quantum permutations and quantum random transpositions.
\newline
\begin{center}
    \textbf{Résumé}
\end{center}

Nous considérons un analogue naturel du mouvement brownien sur les groupes libres quantiques orthogonaux et montrons qu'il a une coupure au temps $N\ln(N)$.
Nous étudions ensuite le processus classique induit sur la droite réelle et calculons ses atomes et sa densité. Cela nous permet de trouver le profil de coupure, qui fait intervenir des lois de Poisson libres et la loi du semi-cercle. Nous prouvons des résultats similaires pour les permutations quantiques et les transpositions aléatoires quantiques.

\section{Introduction}

Let $(X_{N})_{N\in\mathbf{N}}$ be a sequence of irreducible aperiodic finite state Markov chains, $\mu_N(t)$ the distribution of $X_{N}$ after $t$ steps, and $\mu_N(\infty)$ the stationary measure of $X_{N}$. Let also
\begin{equation*}
\mathrm{d}_{N}(t) = \dvt(\mu_N(t),\mu_N(\infty))
\end{equation*}
be the distance of the process to equilibrium at time $t$, where the \emph{total variation distance} $\dvt(\mu,\nu)$ between two probability measures $\mu$ and $\nu$ on a finite set $E$ is defined by the formula
 \begin{equation*}
 \dvt(\mu, \nu) = \frac{1}{2}\sum_{x\in E}\abs{\mu(x)-\nu(x)}.
 \end{equation*}
 Let $(t_{N})_{N\in\mathbf{N}}$ be a sequence of times. We say that $(X_{N})_{N\in\mathbf{N}}$ exhibits a \textit{cutoff} in total variation distance at time $(t_{N})_{N\in\mathbf{N}}$ if for all $\epsilon > 0$,
\begin{equation*}
\mathrm{d}_{N}((1-\epsilon)t_{N}) \xrightarrow[N\rightarrow\infty]{} 1 \quad \text{  and  } \quad  \mathrm{d}_{N}((1+\epsilon)t_{N}) \xrightarrow[N\rightarrow\infty]{} 0.
\end{equation*}
This means that the convergence to equilibrium occurs through a sharp phase transition, falling rapidly from $1$ to $0$ around time $t_{N}$\footnote{We do here (and sometimes in the sequel) a common abuse of notations, not writing the sequence indices. We also do not always write integer parts for random walk times.}.

To get a better understanding of this phenomenon, one may try to zoom in on the window where the ``fall'' occurs. The cutoff phenomenon tells us that the width of this window is negligible with respect to the sequence $(t_{N})_{N\in \N}$, and the next step is therefore to find the next significant ``higher order term''. Here is a way to formalize this. If there exists a sequence $(w_{N})_{N\in\N}$ and a continuous function $f$ decreasing from 1 to 0 such that for all $c\in\R$,
\begin{equation*}
\mathrm{d}_{N}(t_{N} +cw_{N}) \xrightarrow[N\rightarrow\infty]{} f(c),
\end{equation*}
then we say that $f$ is the \textit{cutoff profile} or \textit{limit profile} of $(X_{N})_{N\in\N}$.

Computing the cutoff profile is a difficult task in general, but it could already be done for some important families of Markov chains and commonly involves important probability distributions shaping the profile. For instance, for the lazy random walk on the hypercube (which is equivalent to the Ehrenfest Urn) we have by \cite{Voit1996, coursSalez}
\begin{equation*}
\mathrm{d}_{N}\left(\frac{1}{2}N\ln(N) + cN\right) \xrightarrow[N\rightarrow\infty]{} \dvt\left(\mathcal{N}\left(e^{-c},1\right), \mathcal{N}\left(0,1\right)\right),
\end{equation*}
involving Gaussian distributions. Similar profiles were found for the dovetail shuffle \cite{BayerDiaconis1992}, simple exclusion process on the circle \cite{Lacoin2016}, Ehrenfest Urn with multiple urns \cite{NestoridiThomas}, or Gibbs Sampler \cite{NestoridiThomas}.
For random transpositions, we have by \cite{teyssier2019limit}
\begin{equation*}
\mathrm{d}_{N}\left(\frac{1}{2}(N\ln(N) + cN)\right) \xrightarrow[N\rightarrow\infty]{} \dvt\left(\Poiss\left(1+e^{-c}\right), \Poiss\left(1\right)\right),
\end{equation*}
involving Poisson distributions. The same profile appears also for $k$-cycles \cite{NestoridiThomas}.

This last result on random transpositions, by the second-named author, is one of the motivations of the present article, where we endeavour to compute the cutoff profile for some specific processes. One important difference however is that we will not work with finite classical groups, but with \emph{infinite compact quantum groups}.

Compact quantum groups were introduced by S.L. Woronowicz in \cite{woronowicz1995compact} as a generalization of classical compact groups. In particular, many results from the representation theory of compact groups carry on to this setting, providing tools similar to those used in the study of random transpositions. A recent work of the first-named author \cite{freslon2017cutoff} showed that indeed, there are natural quantum Markov chains on compact quantum groups exhibiting a cutoff phenomenon in a way paralleling the classical case. However, the cutoff profile was not studied there.

In the present paper, we will push further the study of the cutoff phenomenon for stochastic processes on compact quantum groups in two ways. First, we will consider continuous processes instead of discrete ones and second, we will study and describe the cutoff profiles.

The most natural continuous process on a simple compact Lie group is certainly Brownian motion. Recall that this is the process whose diffusion kernel is the heat kernel corresponding to the canonical Riemannian structure on the group. Unfortunately, for quantum analogues of compact Lie groups there is to our knowledge no canonical Riemannian-like structure available to provide an analogue of the heat kernel. However, a result of M. Liao in \cite{liao2004levy} shows that if $(g_{t})_{t\in \R_{+}}$ is a Lévy process on a simple compact Lie group which is invariant under the adjoint action, then its infinitesimal generator is the sum of the Laplace-Beltrami operator (which is the infinitesimal generator of Brownian motion) and a ``jump part'' given by a so-called Lévy measure. It turns out that a similar decomposition also holds for some compact quantum groups. Indeed, F. Cipriani, U. Franz and A. Kula proved in \cite{cipriani2012symmetries} that on the quantum orthogonal group $O_{N}^{+}$, there exists a distinguished process $(\psi_{t})_{t\in \R_{+}}$ such that for any Lévy process which is invariant under the adjoint action, the corresponding infinitesimal generator splits as the sum of the infinitesimal generator of $(\psi_{t})_{t\in \R_{+}}$ and a ``jump part'' characterized by a Lévy measure. As a consequence, $(\psi_{t})_{t\in \R_{+}}$ can be seen as an analogue of Brownian motion.

Our main result is the computation in Section \ref{sec:orthogonal} of the cutoff profile for this Brownian motion on the quantum orthogonal group $O_{N}^{+}$, a compact quantum group which can be thought of as analogue of the group $SO(N)$, for which the cutoff phenomenon was proven by P.-L. Méliot in \cite{meliot2014cut}. More precisely, we prove in Theorem \ref{thm:completeprofileorthogonal} that for any $c \in \R$ and suitable extensions $\widetilde{\mathrm{d}}_{N}$ of the distances $\mathrm{d}_{N}$ to the quantum setting,
\begin{equation*}
\widetilde{\mathrm{d}}_{N} \left(N\ln(N) + cN\right) \xrightarrow[N\rightarrow\infty]{} \dvt\left(\FPoiss\pg e^{2c}, -e^{-c}\pd\ast\delta_{e^{c} + e^{-c}}, \nu_{\mathrm{SC}}\right).
\end{equation*}
where $\nu_{\mathrm{SC}}$ denotes the semi-circle distribution and $\FPoiss$ denotes the free Poisson distribution. It is known that the correspondence between $SO(N)$ (or rather $O(N)$) and $O_{N}^{+}$ has to do, at the probabilistic level, with the Bercovici-Pata bijection \cite{bercovici1999stable}. From that point of view, the appearance of the semi-circle distribution in the cutoff profile is quite satisfying. On the contrary, the appearance of the free Poisson distribution is surprising because it is not a priori a ``deformation'' of the semi-circle distribution. The picture becomes clearer when written in terms of \emph{free Meixner distributions} (see Section \ref{subsec:limit_profile} for the definition) :
\begin{equation*}
\widetilde{\mathrm{d}}_{N} \left(N\ln(N) + cN\right) \xrightarrow[N\rightarrow\infty]{} \dvt\left(\FMeix \pg-e^{-c}, 0\pd\ast\delta_{e^{-c}}, \FMeix(0, 0)\right).
\end{equation*}

Let us briefly comment on the proof. On the one hand, the quantum group $O_{N}^{+}$ is easier to study than $SO(N)$, because its representation theory is simpler (the underlying combinatorics is essentially that of the representation theory of $SU(2)$). This enables to reduce the problem to the study of a classical process on the interval $[-N, N]$. But this is compensated by an analytic issue which is absent from the classical case : the measure associated with the quantum process has atoms as soon as $c < 0$, hence is not absolutely continuous with respect to the limiting distribution. These issues are the translation of a failure of absolute continuity of the process with respect to the Haar measure, which is a purely quantum phenomenon (see Proposition \ref{prop:absolutecontinuity}).
As a consequence, our strategy is first to compute the cutoff profile for $c > 0$ in Proposition \ref{prop:rightprofileorthogonal}, where we have absolute continuity and can therefore reduce the problem to the convergence of the densities, and then to guess from it the form of the cutoff profile for $c<0$. With this in hand and a method inspired from P. Biane in \cite{biane2008introduction}, we are then able to compute the measure of the process also for $c < 0$ and prove the convergence to the cutoff profile in Proposition \ref{thm:negativeprofile}.

In the end of Section \ref{sec:orthogonal}, we investigate other types of convergence and prove that the convergence to the cutoff profile for $c > 0$ also occurs in $L^{p}$-norm for all $1\leqslant p\leqslant \infty$. Let us mention that for $c < 0$, the aforementioned analytic issues enter the picture again, making the very definition of the $L^{p}$-norm problematic, but we nevertheless have convergence of the absolutely continuous part. We furthermore investigate analogues of Brownian motion on some homogeneous spaces of $O_{N}^{+}$ called \emph{free real spheres}, the computations essentially boiling down to the previous ones for $O_{N}^{+}$.

The article concludes in Section \ref{sec:permutations} with a second family of examples called the quantum permutation groups and denoted by $S_{N}^{+}$. Despite bearing strong analogies with the classical permutation group $S_{N}$ justifying its name, $S_{N}^{+}$ is an ``infinite'' compact quantum group. In particular, it has a well-defined Brownian motion, given by a Lévy-Khintchine decomposition similar to that of $O_{N}^{+}$. After computing its cutoff profile, we turn to a problem which was left open in \cite{freslon2017cutoff} : the \emph{quantum random transposition walk}. Here, the absolute continuity issue of the orthogonal case becomes critical : the measure of the corresponding classical process always has an atom so that we cannot resort to densities for $c > 0$. We therefore have to resort to another idea, which is to compare the process with the so-called \emph{pure quantum transposition random walk} and prove that they asymptotically coincide. This is a specifically quantum phenomenon connected to the fact that the pure quantum transposition walk has no periodicity issue because there is no quantum alternating group. More precisely, we show in Theorem \ref{cor:profileparesse} that for $c > 0$,
\begin{equation*}
\widetilde{\mathrm{d}}_{N}\left(\frac{1}{2}(N\ln(N) + cN)\right) \xrightarrow[N\rightarrow\infty]{} \dvt\left( D_{\sqrt{1+e^{-c}}}\left(\FMeix\pg\frac{1-e^{-c}}{\sqrt{1+e^{-c}}}, \frac{-e^{-c}}{1+e^{-c}}\pd\right)\ast\delta_{e^{-c}}, \FMeix(1, 0)\right).
\end{equation*}
As $\FMeix(1, 0) = \FPoiss(1,1)\ast\delta_{-1}$ is the standard free Poisson distribution, this provides a quantum analogue of the result of \cite{teyssier2019limit}.

\subsection*{Acknowledgments}
The authors are indebted to Uwe Franz for pointing out to them the article \cite{biane2008introduction}, the ideas of which helped to improve significantly the results of an earlier version of the present work, and to P.-L. Méliot for discussions on topics linked to the subject of the present paper. A.F. and S.W. were partially funded by the ANR grant ``Noncommutative analysis on groups and quantum groups'' (ANR-19-CE40-0002) and the PHC Polonium ``Quantum structures and processes'', A.F. was also partially funded by the ANR grant ``Operator algebras and dynamics on groups'' (ANR-19-CE40-0008), the PHC Procope ``Quantum groups and quantum probability'' and the PHC Van Gogh ``Quantum groups, harmonic analysis and quantum probability''. S.W. was also partially supported by a public grant as part of the Fondation Mathématique Jacques Hadamard.

\section{Preliminaries}

Compact quantum groups will be one of our main objects of studies in this work, and the one the probabilist reader may be least acquainted with. We will therefore devote this preliminary section to some definitions and fundamental results concerning them. In order to keep things simple, we will only introduce free orthogonal quantum groups for the moment, as well as some results concerning Lévy processes on them. Details on quantum permutation groups will be given when needed later on.

\subsection{Free orthogonal quantum groups}

Free orthogonal quantum groups are examples of compact quantum groups in the sense of S.L. Woronowicz \cite{woronowicz1995compact} which were first introduced by Sh.{} Wang in \cite{wang1995free}. The original definition uses C*-algebras, as may be expected for objects of noncommutative topological nature. We will nevertheless use a different definition which we believe may be easier to understand for the non-expert reader, by focusing first on the purely algebraic aspects. We refer to the books \cite{neshveyev2014compact} and \cite{timmermann2008invitation} for a comprehensive treatment of the theory and proofs of the main results.

\subsubsection{Definition and representation theory}
We recall that a $*$-algebra is an algebra $A$ endowed with an involution $x\mapsto x^*$, i.e. an antimultiplicative linear map such that $(x^* )^* =x$ and $(\lambda x)^* = \bar \lambda x^*$ for all $x\in A$ and $\lambda\in \mathbf C $. Also, a $*$-ideal $B$ of $A$ is a $*$-subalgebra of $A$ such that $\{ba,ab\} \subset B$ for all $a\in A$ and $b\in B$.

\begin{de}
We define $\O(O_{N}^{+})$ to  be the universal $*$-algebra generated by $N^{2}$ \emph{self-adjoint} elements $u_{ij}$ (i.e. $u_{ij}^{*} = u_{ij}$) such that for all $1\leqslant i, j \leqslant N$,
\begin{equation*}
\sum_{k=1}^{N}u_{ik}u_{jk} = \delta_{ij} = \sum_{k=1}^{N}u_{ki}u_{kj}.
\end{equation*}
In other words, 
$$ \O(O_{N}^{+})=\mathbf C \langle u_{ij}:1\leqslant i, j \leqslant N \rangle/I , $$
where $\mathbf C \langle u_{ij}:1\leqslant i, j \leqslant N \rangle$ denotes the $*$-algebra of noncommutative polynomials in variables $u_{ij}, u_{ij}^*$ with $1\leqslant i, j \leqslant N$, and $I$ denotes the $*$-ideal generated by the elements
\begin{equation*}
\ag u_{ij}^{*} - u_{ij} , \sum_{k=1}^{N}u_{ik}u_{jk} - \delta_{ij}, \sum_{k=1}^{N}u_{ki}u_{kj} - \delta_{ij}, 1\leqslant i, j \leqslant N \ad.
\end{equation*}
\end{de}

Let $O_{N}$ be the usual orthogonal group, let $c_{ij} : O_{N} \to \C$ be the function sending a matrix to its $(i, j)$-th coefficient and let $\O(O_{N})$ be the algebra of \emph{regular functions} on $O_{N}$, i.e. the $*$-algebra generated by the functions $c_{ij}$, where the involution corresponds to the complex conjugation : $c_{ij}^{*} = \overline{c_{ij}}$.
Then, quotienting $\O(O_{N}^{+})$ by its commutator ideal yields a surjection
\begin{equation*}
\pi : \O(O_{N}^{+}) \to \O(O_{N})
\end{equation*}
so that $O_{N}^{+}$ can be seen as a ``noncommutative version'' of $O_{N}$. The group structure can be encoded in this setting thanks to the following remark : for any two orthogonal matrices $g$ and $h$,
\begin{equation*}
c_{ij}(gh) =  \sum_{k=1}^{N} c_{ik}(g) c_{kj}(h) = \sum_{k=1}^{N} (c_{ik}\otimes c_{kj})(g,h),
\end{equation*}
where we identify $\O(O_{N}\times O_{N})$ with $\O(O_{N})\otimes \O(O_{N})$. The ``group law'' of $\O(O_{N}^{+})$ will therefore be given by the unique $*$-homomorphism $\D : \O(O_{N}^{+})\to \O(O_{N}^{+})\otimes \O(O_{N}^{+})$, called the \emph{comultiplication}, such that
\begin{equation*}
\D(u_{ij}) = \sum_{k=1}^{N} u_{ik}\otimes u_{kj}.
\end{equation*}
The existence of $\D$ follows from the universal property of $\O(O_{N}^{+})$.

Probability measures can be generalized to this setting by identifying them with their integration linear form. They then correspond to \emph{states}, i.e. linear maps
\begin{equation*}
\psi : \O(O_{N}^{+})\to \C
\end{equation*}
such that $\psi(1) = 1$ and $\psi(x^{*}x) \geqslant 0$ for all $x$. There is a particular state which plays the rôle of the uniform measure on $O_{N}^{+}$ :

\begin{thm}[Woronowicz]
There is a unique state $h$ on $\O(O_{N}^{+})$ such that for all $x\in \O(O_{N}^{+})$,
\begin{equation*}
(\id\otimes h)\circ\D(x) = h(x)\otimes 1 = (h\otimes\id)\circ\D(x).
\end{equation*}
It is called the \emph{Haar state} of $O_{N}^{+}$.
\end{thm}

Since the founding works of P. Diaconis and his coauthors, it is known that representation theory is a powerful tool to study the asymptotic behaviour of random walks on groups (see for instance \cite[Chap 4]{diaconis1988group}). For $O_{N}^{+}$, the representation theory was computed by T. Banica in \cite{banica1996theorie}. However, for our purpose we will only need to understand the subalgebra $\O(O_{N}^{+})_{\mathrm{central}}$ generated by the characters of the irreducible representations (we refer the reader for instance to \cite[Sec 1.3]{neshveyev2014compact} for the definitions of these notions and details).

\begin{thm}[Banica]
Let us set $\chi_{0} = 1$ and $\chi_{1} = \sum_{i=1}^{N} u_{ii}$. Then, the irreducible representations of $O_{N}^{+}$ are labelled by the integers such that if $\chi_{n}$ denotes the character associated to the integer $n\in\N$, we have the recurrence relation :
\begin{equation*}
\forall n \geqslant 1, \quad \chi_{1}\chi_{n} = \chi_{n+1} + \chi_{n-1}.
\end{equation*}
\end{thm}

Note that this implies that $\chi_{n}^{*} = \chi_{n}$ for all $n\in \N$. This recurrence relation is reminiscent of Chebyshev polynomials, and one can indeed express $\chi_{n}$ in terms of $\chi_{1}$ using them. More precisely, let $(P_{n})_{n\in \N}$ be the sequence of polynomials defined by $P_{0}(X) = 1$, $P_{1}(X) = X$ and
\begin{equation*}
XP_{n}(X) = P_{n+1}(X) + P_{n-1}(X).
\end{equation*}
In particular, $P_n(2) = n+1$, $P_n(-2) = (-1)^n(n+1)$, and for $\theta \in ]0,\pi[$ and $n\in\mathbf{N}$,
\begin{equation*}
P_n(2\cos(\theta)) = \frac{\sin((n+1)\theta)}{\sin(\theta)}.
\end{equation*}
Then, the map $\iota:\chi_{n}\mapsto P_{n}$ yields an isomorphism between $\O(O_{N}^{+})_{\mathrm{central}}$ and $\C[X]$. Moreover, by \cite[Prop 1]{banica1997groupe}, the restriction of the Haar state to this subalgebra coincides with integration with respect to the semicircle distribution $\nu_{\mathrm{sc}}$. More precisely, for any $x\in \O(O_{N}^{+})_{\mathrm{central}}$ we have
\begin{equation*}
h(x)=\int_{[-2,2]}\iota(x)\mathrm{d}\nu_{\mathrm{sc}},
\end{equation*}
where $\nu_{\mathrm{sc}}$ denotes the measure on [-2,2] with density
\begin{equation*}
\mathrm{d}\nu_{\mathrm{sc}} = \frac{1}{2\pi}\sqrt{4-x^{2}}\mathrm{d}x.
\end{equation*}
The polynomials $P_{n}$ are exactly the orthogonal polynomials for this measure. Moreover, denoting by $d_n$ the dimension of the $n$-th irreducible representation, we have for $n\in \mathbf N$,
\begin{equation*}
d_n = P_n (N).
\end{equation*}

\subsubsection{Central Lévy processes}

Let us now describe what the analogue of a Lévy process is on $O_{N}^{+}$. On a classical group, this is a càdlàg stochastic process $(X_{t})_{t\in\R_{+}}$ with independent and stationary increments. In particular, if $\mu_{t}$ is the distribution of $X_{t}X_{0}^{-1}$ then we have
\begin{itemize}
\item $\mu_{0} = \delta_{\mathrm{Id}}$,
\item $\mu_{t}\ast\mu_{s} = \mu_{t+s}$,
\item $\displaystyle\lim_{t\to 0} \mu_{t} = \mu_{0}$ weakly. 
\end{itemize}
In other words, we have a right-continuous convolution semigroup of probability measures. Because this semigroup contains most of the probabilistic information about the process, and in particular concerning its asymptotic behaviour, we will focus on it and define a \emph{quantum Lévy process} to be a right-continuous convolution semigroup of states, i.e. a family $(\psi_{t})_{t\in \R_{+}}$ of states on $\O(O_{N}^{+})$ such that
\begin{itemize}
\item $\psi_{0} = \varepsilon : u_{ij}\mapsto \delta_{ij}$,
\item $\psi_{t}\ast\psi_{s} = (\psi_{t}\otimes \psi_{s})\circ\D = \psi_{t+s}$,
\item $\displaystyle\lim_{t\to 0}\psi_{t}(x) = \psi_{0}(x)$ for all $x\in \O(O_{N}^{+})$.
\end{itemize}
Let us mention that the theory of Lévy processes on compact quantum groups can be developped in full generality (not just restricted to marginals), see for instance the survey \cite{franz2006levy}.

As mentioned in the introduction, we will be interested in the case where the Lévy process is invariant under the adjoint action. In other words we will focus on states which are \emph{central}, in the sense that they are invariant by conjugation (see the beginning of \cite[Sec 6]{cipriani2012symmetries} for the definitions). By \cite[Prop 6.9]{cipriani2012symmetries}, such states have a specific form. First, there is a \emph{conditional expectation}
\begin{equation*}
\E : \O(O_{N}^{+})\to \O(O_{N}^{+})_{\mathrm{central}}\;,
\end{equation*}
that is to say a linear map satisfying $\E(x^{*}x)\geqslant 0$ for all $x$, and $\E(x) = x$ for all $x\in\O(O_{N}^{+})_{\mathrm{central}}$. Then, a state $\psi$ is central if and only if there exists a state $\widetilde{\psi}$ on $\O(O_{N}^{+})_{\mathrm{central}}$ such that
\begin{equation*}
\psi = \widetilde{\psi}\circ\E.
\end{equation*}

\begin{de}
A \emph{central Lévy process} on $O_{N}^{+}$ is a continuous convolution semigroup of states $(\psi_{t})_{t\in \R_{+}}$ on $\O(O_{N}^{+})$ such that $\psi_{t}$ is central for all $t$.
\end{de}

Central Lévy processes on $O_{N}^{+}$ were classified by F. Cipriani, U. Franz and A. Kula in \cite[Thm 10.2]{cipriani2012symmetries} through an analogue of the Lévy-Khinchine formula. Note that because of centrality, it is enough to know the image of $\chi_{n}$ for all $n\in \N$.

\begin{thm}[Cipriani-Franz-Kula]
Any central Lévy process $(\psi_{t})_{t\in \R_{+}}$ on $O_{N}^{+}$ is of the form
\begin{equation*}
\psi_{t} :\chi_{n}\mapsto P_{n}(N)e^{-t\lambda_n},
\end{equation*}
where
\begin{equation}\label{eq:quantumhunt}
\lambda_n = b\frac{P_{n}'(N)}{P_{n}(N)} + \frac{1}{P_{n}(N)}\int_{-N}^{N}\frac{P_{n}(N) - P_{n}(x)}{N-x}\mathrm{d}\nu(x)
\end{equation}
for some $b \geqslant 0$ and a finite measure $\nu$ on $[-N, N]$ such that $\nu(\{N\}) = 0$.
\end{thm}

Comparing this formula with the one proved by M. Liao for classical compact Lie groups in \cite{liao2004levy}, we see that the process corresponding to $\lambda_n = P_{n}'(N)/P_{n}(N)$ plays a role analogous to the one associated to the Laplace-Beltrami operator. As a consequence, we will call this process the \emph{Brownian motion} on $O_{N}^{+}$.

\subsection{The cutoff phenomenon}\label{sec:preliminariescutoff}

This work is mostly concerned with the diffusion of central Lévy processes and in particular the time needed for the process to spread all over the group. This can be rigorously defined by measuring the distance between $\psi_{t}$ and the Haar state $h$. Classically, one interesting and widely used distance for this is the \emph{total variation distance}
\begin{equation*}
\dvt(\mu,\nu) = \|\mu - \nu\|_{TV} = \sup_{A\subset G}\vert\mu(A) - \nu(A)\vert,
\end{equation*}
where the supremum is taken over all Borel subsets $A$ of the classical group $G$. For quantum groups, the corresponding definition requires the introduction of a suitable version of the Borel $\sigma$-algebra.

To this end, one may consider the universal enveloping C*-algebra (see for instance \cite[Sec II.8.3]{blackadar2006operator}) $C(O_{N}^{+})$ of $\O(O_{N}^{+})$. By definition, any state on $\O(O_{N}^{+})$ has a unique extension to a state on $C(O_{N}^{+})$, hence yields an element of the \emph{Fourier-Stieltjes algebra}, which is the topological dual $C(O_{N}^{+})^{*}$ of $C(O_{N}^{+})$. In view of the Riesz representation theorem, this dual space is thought of as a noncommutative analogue of the measure algebra equipped with the total variation norm. Let us denote by $\|\cdot\|_{FS}$ the norm on this dual space and call it the \emph{Fourier-Stieltjes norm}. Moreover, the topological double dual $C(O_{N}^{+})^{**}$ of $C(O_{N}^{+})$ is known to be the \emph{universal enveloping von Neumann algebra} of $C(O_{N}^{+})$, which is regarded as a noncommutative and universal analogue of measure spaces. Using the theory of Haagerup's noncommutative $L^p$-spaces, we may easily adapt the argument in \cite[Lemma 2.6]{freslon2017cutoff} to show that 
\begin{equation*}
\frac{1}{2} \|\varphi - \psi\|_{FS} = \sup_{p\in \mathcal{P}}\vert \varphi(p) - \psi(p)\vert,
\end{equation*}
where $\mathcal{P}$ denotes the set of orthogonal projections in $C(O_{N}^{+})^{**}$ (thought of as indicator functions of Borel subsets). We omit the details of the proof since we do not need it in this paper. We will use this generalized total variation distance in our cutoff statements and write :
\begin{equation*}
\|\cdot\| = \frac{1}{2}\|\cdot\|_{FS}.
\end{equation*}

In the classical setting, a particularly important case is that both $\mu$ and $\nu$ are absolutely continuous with respect to the Haar measure. We may consider the similar situation in the quantum setting. Let us define an inner product on $\O(O_{N}^{+})$ by the formula $\langle x, y\rangle = h(xy^{*})$.
Then, taking the completion yields a Hilbert space $L^{2}(O_{N}^{+})$, and $\O(O_{N}^{+})$ embeds through left multiplication into $\B(L^{2}(O_{N}^{+}))$ (see \cite[Cor 1.7.5]{neshveyev2014compact} and the comments thereafter). The weak closure of the image is denoted by $L^{\infty}(O_{N}^{+})$ and is a von Neumann algebra. If $\varphi: \O(O_{N}^{+})\to \C$ is a linear map which extends to a normal bounded map on $L^{\infty}(O_{N}^{+})$, then $\varphi$ becomes an element of the \emph{Fourier algebra}, which is the Banach space predual $L^{\infty}(O_{N}^{+})_{*}$ of $L^{\infty}(O_{N}^{+})$ and $\|\varphi\|_{FS} = \|\varphi\|_{L^{\infty}(O_{N}^{+})_*}$ (see for instance \cite[Prop 3.14]{brannanruan2017lp}), which further implies  by \cite[Lem 2.6]{freslon2017cutoff}: 
\begin{equation*}
  \|\varphi - \psi\| = \sup_{p\in \tilde{ \mathcal{P}}}\vert \varphi(p) - \psi(p)\vert,
\end{equation*}
where $\tilde{ \mathcal{P}}$ is the set of orthogonal projections in $L^{\infty}(O_{N}^{+})$. Note that in order for this formula to make sense, the states $\varphi$ and $\psi$ must extend to the von Neumann algebra $L^{\infty}(O_{N}^{+})$. This is not always the case due to absolute continuity issues (see for instance Proposition \ref{prop:absolutecontinuity}). Let us mention an elementary but useful fact on the monotonicity of that norm which is well-known in the classical case.

\begin{lem}\label{lem:decreasing}
For $N$ fixed, the map $t\mapsto \|\psi_{t} - h\|$ is decreasing.
\end{lem}

\begin{proof}
First note that for any two bounded linear forms $\varphi, \psi$ on $C(O_{N}^{+})$,
\begin{equation*}
\|\varphi\ast\psi\|_{FS} = \|(\varphi\otimes \psi)\circ\D\|_{FS}\leqslant \|\varphi\otimes \psi\|_{FS} \leqslant \|\varphi\|_{FS}\|\psi\|_{FS}.
\end{equation*}
Thus, for any $t\geqslant s$,
\begin{equation*}
\|\psi_{t} - h\|_{FS} = \|(\psi_{s} - h)\ast\psi_{t-s}\|_{FS}\leqslant \|\psi_{s} - h\|_{FS},
\end{equation*}
where we used the fact that any state $\psi$ on a C*-algebra has norm one (see for instance \cite[Lem 9.9]{takesaki2002theory}). The result follows.
\end{proof}

The evolution of the distance from the process to the Haar state can exhibit various behaviours. One which is especially striking is the so-called \emph{cutoff phenomenon}. Here is a precise definition of what we mean by this :

\begin{de}\label{de:cutoff}
Let $(\G_{N}, (\psi_{t}^{(N)})_{t\in \R^{+}})_{N\in \N}$ be a family of compact quantum groups with a Lévy process $(\psi_{t}^{(N)})_{t\in \R^{+}}$ on each of them. We say that the processes exhibit a cutoff phenomenon at time $(t_{N})_{N\in \N}$ if for any $\epsilon > 0$,
\begin{equation*}
\lim_{N\to +\infty} \|\psi_{(1-\epsilon)t_{N}}^{(N)} - h_{N}\| = 1 \text{ and } \lim_{N\to +\infty} \|\psi_{(1+\epsilon)t_{N}}^{(N)} - h_{N}\| = 0.
\end{equation*}
\end{de}

One very useful tool to prove that such a phenomenon occurs is the following lemma originally due to P. Diaconis and M. Shahshahani in \cite{diaconis1981generating} for finite groups and to J.P. McCarthy in \cite{mccarthy2018diaconis} for finite quantum groups. A proof for compact quantum groups can be found in \cite[Lem 2.7]{freslon2017cutoff}, but we simply state it in our particular case.

\begin{lem}\label{lem:upperbound}
Let $\psi$ be a central state on $O_{N}^{+}$. If for some $t>0$ the sum $\sum_{n=1}^{+\infty}d_n^{2}e^{-2t\lambda_n} $ is finite, then
\begin{equation}\label{eq:upperbound}
\|\psi_{t} - h\|^{2} \leqslant \frac{1}{4}\sum_{n=1}^{+\infty}P_{n}(N)^{2}e^{-2t\lambda_n} .
\end{equation}
\end{lem}

\begin{rem}
The right-hand side is nothing but the $L^{2}$-norm of the density of $\psi_{t}-h$ with respect to $h$, computed using the Plancherel formula. We recover in that way the fact that as soon as a state has an $L^{2}$-density with respect to the Haar state, it has an extension to $L^{\infty}(O_{N}^{+})$.
\end{rem}

Let us end these preliminaries with some computational results concerning the Chebyshev polynomials introduced above. First, there is an explicit formula to compute these numbers for $t > 2$ : setting
\begin{equation*}
q(t) = \frac{t - \sqrt{t^{2} - 4}}{2},
\end{equation*}
it is easily checked by induction that
\begin{equation*}
P_{n}(t) = \frac{q(t)^{-(n+1)} - q(t)^{n+1}}{q^{-1}-q}.
\end{equation*}
This enables us to give a precise expansion of the polynomials $P_{n}$.

\begin{prop}\label{prop:dlpn}
For any $n\geqslant 1$, we have as $t\to\infty$
\begin{equation*}
P_{n}(t) = t^{n}\pg 1 - \frac{1}{t^{2}} + O\pg\frac{1}{t^{4}}\pd\pd^{n-1}.
\end{equation*}
\end{prop}

\begin{proof}
Let $n\geqslant 1$. First, as $t\to\infty$, we have
\begin{equation*}
q(t) = \frac{t-\sqrt{t^{2}-4}}{2} = \frac{1}{t} + \frac{1}{t^{3}} + O\pg\frac{1}{t^{5}}\pd \text{ and } q(t)^{-1} = t - \frac{1}{t} + O\pg \frac{1}{t^{3}} \pd.
\end{equation*}
We deduce, writing $q$ for $q(t)$, that for $t\geqslant 2$ we have 
\begin{equation*}
    \ln(P_n(t)) = \ln\pg \frac{q^{-(n+1)}}{q^{-1}}\frac{1-q^{2(n+1)}}{1-q^2} \pd = n\ln(q^{-1}) + \ln\pg 1-q^{2(n+1)} \pd -\ln \pg 1-q^{2} \pd.
\end{equation*}
Note that $-\ln(1-q^2) = q^2 + O\pg q^4\pd = \frac{1}{t^2} + O\pg \frac{1}{t^4}\pd$ and $2(n+1) \geqslant 4$ so $\ln\pg 1-q^{2(n+1)} \pd = O\pg \frac{1}{t^4}\pd$, hence
\begin{align*}
    \ln(P_n(t)) = n\pg \ln(t) -\frac{1}{t^2} +O\pg \frac{1}{t^4} \pd\pd + O\pg \frac{1}{t^4}\pd + \frac{1}{t^2} + O\pg \frac{1}{t^4}\pd.
\end{align*}
Hence, putting the three $O\pg \frac{1}{t^4}\pd$ together (it might change the implicit constant, but we can still take it to be independent of $n$), we finally have
\begin{equation*}
    \ln(P_n(t)) = n\ln(t) + (n-1)\pg -\frac{1}{t^2} + O\pg \frac{1}{t^4}\pd \pd,
\end{equation*}
i.e.
\begin{equation*}
P_n(t) = t^n \exp\pg -\frac{1}{t^2} + O\pg \frac{1}{t^4}\pd \pd^{n-1} = t^n \pg 1 -\frac{1}{t^2} + O\pg \frac{1}{t^4}\pd \pd^{n-1}.
\end{equation*}
\end{proof}

\section{The quantum orthogonal brownian motion}\label{sec:orthogonal}

In this section we study the cutoff phenomenon for the analogue of Brownian motion on $O_{N}^{+}$. As explained above, this means that we will take $\nu = 0$ in Equation \eqref{eq:quantumhunt}. Once that choice is made, changing the value of $b$ is equivalent to rescaling the time, so that there is no loss in generality in fixing $b = 1$ for all $N$, leading to
\begin{equation*}
\lambda_n = \frac{P_{n}'(N)}{P_{n}(N)}.
\end{equation*}

\subsection{The cutoff phenomenon}\label{subsec:cutofforthogonal}

We first want to prove that the process $(\psi_{t})_{t\in \R_{+}}$ exhibits a cutoff phenomenon. It turns out that this result can be recovered as a consequence of the existence of a cutoff profile, proven in Theorem \ref{thm:completeprofileorthogonal}, as will be explained in Remark \ref{rem:profileimpliescutoff} below. Nevertheless, the proof below is independent of the computations of the cutoff profile and is interesting in its own right. One reason for this is that, in particular as far as the lower bound is concerned, the argument below fills a gap in the proof of previous similar theorems in \cite{freslon2017cutoff} and \cite{freslon2018quantum} (see in particular Remark \ref{rem:truecutoff}). Another reason is that to compute the profile, one must know the cutoff time and guess what the next order is, which we do when proving the cutoff phenomenon. That being said, let us give some additional details concerning the Fourier-Stieltjes norm needed for the proof.

As explained in Section \ref{sec:preliminariescutoff}, proving a cutoff phenomenon requires precise estimates for the Fourier-Stieltjes norm. As soon as the right-hand side of Equation \eqref{eq:upperbound} is finite, we can use it to bound the total variation distance, which is then well-defined and coincides with the half of the Fourier-Stieltjes norm. This is the strategy which was already used in \cite{freslon2017cutoff} and the computations will be similar. To obtain the lower bound, however, we need to deal directly with the Fourier-Stieltjes norm and this will require an alternate description which we now detail.

By \cite[Lem 4.2]{brannan2011approximation}, the closure of $\O(O_{N}^{+})_{\text{central}}$ in $C(O_{N}^{+})$ is a commutative C*-algebra isomorphic to $C([-N, N])$. Moreover, if $\varphi = \widetilde{\varphi}\circ\E$ is a bounded central linear form, then
\begin{equation*}
\|\widetilde{\varphi}\|_{FS} = \left\|\varphi_{\vert \overline{\O(O_{N}^{+})_{\text{central}}}}\right\|_{FS} \leqslant \|\varphi\|_{FS} = \|\widetilde{\varphi}\circ\E\|_{FS}\leqslant \|\widetilde{\varphi}\|_{FS}
\end{equation*}
so that the problem reduces to the computation of the norm of a bounded linear form on a commutative C*-algebra. By the Riesz Representation Theorem, there exists a measure $\mu$ on $[-N, N]$ such that
\begin{equation*}
\widetilde{\varphi}(x) = \int_{-N}^{N}x\mathrm{d}\mu
\end{equation*}
and moreover, the Fourier-Stieltjes norm of $\widetilde{\varphi}$ coincides with twice the total variation of $\mu$. Using that observation, we can now establish that Brownian motion on $O_{N}^{+}$ exhibits a cutoff phenomenon.

\begin{thm}\label{thm:orthogonalcutoff}
Brownian motion on $O_{N}^{+}$ exhibits a cutoff phenomenon at time $t_{N} = N\ln(N)$.
\end{thm}

\begin{proof}
We start with the upper bound and we will obtain an estimate which is finer than what is actually needed. It was proven in \cite[Lem 1.7]{franz2017hypercontractivity} that (we will give in Lemma \ref{lem:affineestimate} on a finer estimate, but that one is sufficient for our present purpose)
\begin{equation*}
\frac{n}{N}\leqslant \lambda_n\leqslant \frac{n}{N-2}.
\end{equation*}
Using this and the estimates of \cite[Lem 3.3]{freslon2017cutoff}, and writing $q=q(N)$ for simplicity, the sum in the right-hand side of Lemma \ref{lem:upperbound} applied to $\psi_{t}$ can be bounded as soon as $q^{-1}e^{-t/N} < 1$ by
\begin{equation*}
\sum_{n=1}^{+\infty}\frac{q^{-2n}}{(1-q^{2})^{2}}e^{-2tn/N} = \frac{1}{(1-q^{2})^{2}}\frac{q^{-2}e^{-2t/N}}{1 - q^{-2}e^{-2t/N}} =  \frac{1}{(1-q^{2})^{2}}\frac{1}{q^{2}e^{2t/N} - 1}.
\end{equation*}
For $c > 0$ and $t = N\ln(N) + cN$, we get, using $1/2 > q(N) > 1/N$ (see for instance \cite[Lem 3.8]{freslon2017cutoff}),
\begin{align*}
\frac{1}{(1-q^{2})^{2}}\frac{1}{q^{2}N^{2}e^{2c} - 1} & \leqslant \frac{4}{3}\frac{e^{-2c}}{1-e^{-2c}}.
\end{align*}
Taking $c = \epsilon\ln(N)$ with $\epsilon > 0$ then yields the upper bound part of the cutoff phenomenon.

For the lower bound, we will show that the character $\chi_{1}$ is a good witness of the distance between $\psi_{t}$ and $h$. To do that, let us first estimate its mean and variance. We have, for $c < 0$ and $t = N\ln(N) + cN$,
\begin{equation*}
\psi_{t}(\chi_{1}) = Ne^{-t/N} = e^{-c}
\end{equation*}
and the variance $\var_{\psi_{t}}(\chi_{1}) = \psi_{t}(\chi_{1}^{2}) - \psi_{t}(\chi_{1})^{2}$ is given, using the fact that $\chi_{1}^{2} = 1 + \chi_{2}$, by
\begin{align*}
\var_{\psi_{t}}(\chi_{1}) & = 1 + (N^{2}-1)e^{-2tN/(N^{2}-1)} - N^{2}e^{-2t/N} \\
& \leqslant 1 + N^{2}e^{-2t/N} - N^{2}e^{-2t/N} \\
& \leqslant 1 .
\end{align*}
Let us now view $\chi_{1}$ as a continuous function on $[-N, N]$ and consider the Borel subset
\begin{equation*}
B = \{s\in [-N, N] \mid \vert\chi_{1}(s)\vert\leqslant e^{-c}/2\}.
\end{equation*}
The indicator function $p = \textbf{1}_{B}$ can be seen as a projection in the von Neumann algebra $L^{\infty}([-N, N])$ of essentially bounded functions. If we denote by $\nu_{h}$ (respectively $\mu_{t}$) the unique Borel probability measure on $[-N, N]$ such that for any $x\in \O(O_{N}^{+})_{\text{central}}$,
\begin{equation*}
h(x) = \int_{-N}^{N}x\mathrm{d}\nu_{h} \quad\quad \left( \text{respectively }\quad \psi_{t}(x) = \int_{-N}^{N}x\mathrm{d}\mu_{t}\right),
\end{equation*}
then these formul\ae{} provide norm-preserving extensions of the states $h$ and $\psi_{t}$ to $L^{\infty}([-N, N])$. Moreover, because $\psi_{t}(\chi_{1}) = e^{-c}$, we have $B\subset\{s\in [-N, N] \mid \vert \chi_{1}(s) - \psi_{t}(\chi_{1})\vert\geqslant e^{-c}/2\}$ so that by Chebyshev's inequality
\begin{equation*}
\mu_{t}(B)\leqslant \left(e^{-c}/2\right)^{-2}\mathrm{var}_{\mu_{t}}(\chi_1) \leqslant 4e^{2c}.
\end{equation*}
Using again Chebyshev's inequality for $\nu_{h}$ with $h(\chi_{1}) = 0$ and $\var_{h}(\chi_{1}) = 1$, we eventually get
\begin{align*}
\vert\mu_{t}(B) - \nu_{h}(B)\vert & \geqslant \nu_{h}(B) - \mu_{t}(B) \\
& = 1 - \nu_{h}([-N, N]\setminus B) - \mu_{t}(B) \\
& \geqslant 1 - 4e^{2c} - 4e^{2c} \\
& = 1 - 8e^{2c}.
\end{align*}
To conclude, recall that because $\mu_{t}$ and $\nu_{h}$ are probability measures, the total variation norm of their difference coincides with twice their total variation distance, so that
\begin{equation*}
\|\psi_{t} - h\| = \frac{1}{2}\vert\mu_{t} - \nu_{h}\vert([-N, N]) = \|\mu_{t} - \nu_{h}\|_{TV}\geqslant \vert\mu_{t}(B) - \nu_{h}(B)\vert \geqslant 1 - 8e^{2c}.
\end{equation*}
For $c = -\epsilon\ln(N)$ ($\epsilon > 0$ fixed), the right-hand side becomes $1 - 8N^{-2\epsilon}$ which tends to $1$ as $N \to +\infty$, hence the proof is complete.
\end{proof}

\begin{rem}\label{rem:truecutoff}
In the papers \cite{freslon2017cutoff} and \cite{freslon2018quantum}, the lower bounds were only proven for $c$ such that the corresponding state is absolutely continuous with respect to the Haar state. 
As a consequence, this does not yield a cutoff phenomenon in the sense of Definition \ref{de:cutoff} unless one makes sure that the states are asymptotically always absolutely continuous.
As we will show in Proposition \ref{prop:absolutecontinuity}, and this was already observed in the aforementioned papers, this is never true. Hence, the term ``cutoff'' was slightly abusive there. However, using the same argument as in the above proof together with the estimates of \cite{freslon2017cutoff} and \cite{freslon2018quantum}, one can easily show that the random walks studied there indeed exhibit a \emph{bona fide} cutoff phenomenon for the Fourier-Stieltjes norm.
\end{rem}

\begin{rem}
In \cite{meliot2014cut}, P.-L. Méliot proved that Brownian motion on $SO(N)$ exhibits a cutoff phenomenon at time $2\ln(N)$. The factor $2$ comes from the fact that he chooses one half of the Laplace-Beltrami operator as an infinitesimal generator. As for the additional factor $N$ in our result, it could be removed through setting $b_{N} = N$. A scaling-free statement on our case would therefore be that the cutoff time $t_{N}$ satisfies $t_{N}b_{N} = N\ln(N)$ while in the case of P.-L. Méliot we have $t_{N}b_{N} = \ln(N)$.
\end{rem}

\begin{rem}
Thanks to the work of F. Cipriani, U. Franz and A. Kula \cite{cipriani2012symmetries}, it is possible to construct a non-commutative Riemannian structure (a \emph{spectral triple}) on $O_{N}^{+}$ out of a Lévy process. However, it was already noted in \cite[Sec 10]{cipriani2012symmetries} that in our case, and independently from the choice of $b$, the dimension of the resulting object is infinite, while one would expect a canonical Remiannian-like structure to be able to recover $N$ through a notion of dimension.
\end{rem}

We mentioned earlier that the use of the Fourier-Stieltjes norm was necessary because $\psi_{t}$ cannot be extended to $L^{\infty}(O_{N}^{+})$ in general. This can be thought of as an absolute continuity issue in the following sense. Let us denote by $L^{1}(O_{N}^{+})$ the completion of $L^{\infty}(O_{N}^{+})$ with respect to the norm $\|x\|_{1} = h(\vert x\vert)$, where $\vert x\vert$ is obtained by functional calculus. A state $\psi_{t}$ is then said to be \emph{absolutely continuous} (with respect to the Haar state) if there exists $\rho_{t}\in L^{1}(O_{N}^{+})$ such that $\psi_{t}(x) = h(\rho_{t}x)$ for all $x\in \O(O_{N}^{+})$. It follows from the general theory (see \cite[Thm V.2.18]{takesaki2002theory}) that a state on $\O(O_{N}^{+})$ is absolutely continuous with respect to the Haar state if and only if it extends to a normal linear map on $L^{\infty}(O_{N}^{+})$.

We now want to give a precise result about absolute continuity, and this requires a finer estimate on $\lambda_{n}$ than that of \cite[Lem 1.7]{franz2017hypercontractivity}. In fact, we will prove that $\lambda_{n}$ is very close to being affine (and even linear).

\begin{lem}\label{lem:affineestimate}
Set ${a_N} = 1/\sqrt{N^{2} - 4}$ and ${b_N} = (N-\sqrt{N^{2}-4})/(N^{2}-4)$.
Then, for any $n\in \N$ and $N\geqslant 4$,
\begin{equation*}
\lambda_n - ({a_N}n + {b_N}) = {c_N(n)} := \frac{1}{N}\sum_{j=0}^{+\infty}\left(\frac{2}{N}\right)^{2j}\left( \frac{n+1}{2^{2j}}\sum_{\underset{\ell \equiv j \text{ mod } n+1}{\ell = j+1}}^{2j}\binom{2j}{\ell}\right),
\end{equation*}
and moreover,
\begin{equation*}
0\leqslant {c_N(n)} \leqslant \left(\frac{2}{N}\right)^{n}.
\end{equation*}
\end{lem}

\begin{proof}
Let $n\in\N$ and $N\geq 4$. Recall that the roots of $P_{n}$ are $x_{k} = 2\cos(k\pi/(n+1))$ for $1\leqslant k\leqslant n$, so that

\begin{equation*}
\lambda_n = \frac{P_{n}'(N)}{P_{n}(N)} = \sum_{k=1}^{n}\frac{1}{N-x_{k}} = \sum_{k=1}^{n}\frac{1}{N}\sum_{j=0}^{+\infty}\left(\frac{x_{k}}{N}\right)^{j} = \frac{1}{N}\sum_{j=0}^{+\infty}\left(\frac{2}{N}\right)^{j}\sum_{k=1}^{n}\cos\left(\frac{k\pi}{n+1}\right)^{j}.
\end{equation*}
We are therefore led to compute some sums of powers of trigonometric functions.

Let $j\in\mathbf{N}$ and set for convenience in the next computations $\theta = \frac{\pi}{n+1}$. Then we have,
\begin{align*}
\sum_{k=0}^{n}\cos(k\theta)^{j} = \sum_{k=0}^{n} \left(\frac{e^{ik\theta}+e^{-ik\theta}}{2}\right)^{j} = \sum_{k=0}^{n}\frac{1}{2^{j}}\sum_{\ell = 0}^{j}\binom{j}{\ell} e^{ik\theta(2\ell - j)} = \frac{1}{2^{j}}\sum_{\ell = 0}^{j}\binom{j}{\ell} \sum_{k=0}^{n} e^{ik\theta(2\ell - j)}.
\end{align*}
Now observe that 
\begin{equation*}
\sum_{k=0}^{n} e^{ik\theta(2\ell - j)} = \left\{
\begin{array}{ll}
0 & \mbox{if } 2(n+1) \nmid 2\ell - j \\
n+1 & \mbox{if } 2(n+1) \mid 2\ell - j.
\end{array}
\right.
\end{equation*}
From this we see immediately that if $j$ is odd, then the sum vanishes. If $j$ is even, it is possible that $2(n+1) \mid 2\ell - j$ (i.e. that $n+1 \mid \ell - j/2$), and we have to consider all such $\ell$'s. We can hence rewrite the sum as
\begin{equation*}
\sum_{k=0}^{n}\cos(k\theta)^{j} = \frac{n+1}{2^{j}}\sum_{\underset{\ell \equiv j/2 \text{ mod } n+1}{\ell = 0}}^{j}\binom{j}{\ell}.
\end{equation*}
We deduce, splitting the previous sum according to whether $\ell = j/2$ or not, that 
\begin{equation*}
\lambda_{n} + \frac{1}{N}\sum_{j=0}^{+\infty}\left(\frac{2}{N}\right)^{2j} = \frac{n+1}{N}\sum_{j=0}^{+\infty}\left(\frac{1}{N}\right)^{2j}\binom{2j}{j} + \frac{1}{N}\sum_{j=0}^{+\infty}\left(\frac{2}{N}\right)^{2j}\left( \frac{n+1}{2^{2j}}\sum_{\underset{\ell \equiv j \text{ mod } n+1}{\ell = j+1}}^{2j}\binom{2j}{\ell}\right),
\end{equation*}
i.e., observing that $(1-4x)^{-1/2} = \summ{j=0}{\infty} \binom{2j}{j}x^j$ for $\abs{x} < 1/4$,
\begin{equation*}
\lambda_{n} + \frac{1}{N}\frac{1}{1-4/N^2} = \frac{n+1}{N}\frac{1}{\sqrt{1-4/N^2}} + {c_N(n)},
\end{equation*}
which rewrites exactly as
\begin{equation*}
\lambda_{n} = {a_N}n + {b_N} + {c_N(n)}.
\end{equation*}
Let us now bound ${c_N(n)}$.
Because the sum $\sum_{\ell = j+1, \ell\equiv j \mod n+1}^{2j}$ is empty if $j < n$, we get
\begin{equation*}
{c_N(n)} = \frac{1}{N}\sum_{j=0}^{+\infty}\left(\frac{2}{N}\right)^{2j}\left( \frac{n+1}{2^{2j}}\sum_{\underset{\ell \equiv j \text{ mod } n+1}{\ell = j+1}}^{2j}\binom{2j}{\ell}\right) \leqslant \frac{1}{N}\sum_{j=n}^{+\infty}\left(\frac{2}{N}\right)^{2j}\left( \frac{n+1}{2^{2j}} 2^{2j} \right) = \frac{n+1}{N-2}\left(\frac{2}{N}\right)^{2n}.
\end{equation*}
As $N\geq 4$, we finally obtain
\begin{equation*}
    {c_N(n)} \leqslant \frac{n+1}{N-2}\left(\frac{2}{N}\right)^{2n} \leqslant \frac{1}{N-2}\frac{n+1}{2^n}\pg\frac{2}{N}\pd^n \leqslant \pg\frac{2}{N}\pd^n.
\end{equation*}
\end{proof}

We can now give a precise criterion for absolute continuity.

\begin{prop}\label{prop:absolutecontinuity}
Let $N \geqslant 4$. Then, there exists a positive time $t_{abscont}(N)$ such that
\begin{itemize}
\item If $t < t_{abscont}(N)$, then $\psi_{t}$ is not absolutely continuous with respect to the Haar state,
\item If $t > t_{abscont}(N)$, then $\psi_{t}$ is absolutely continuous with respect to the Haar state. \\
Moreover,
\begin{equation*}
t_{abscont}(N) = \frac{-\ln(q(N))}{a_{N}} = N\ln(N) - \frac{2\ln(N)}{N} + \underset{N\to+\infty}{O}\pg \frac{1}{N} \pd.
\end{equation*}
\end{itemize}
\end{prop}

\begin{proof}
Let $N \geqslant 4$, and set
\begin{equation*}
\rho_{t} = \sum_{n=0}^{+\infty}d_{n}e^{-t\lambda_n}\chi_n.
\end{equation*}
If this series converges in $L^{1}(O_{N}^{+})$, then $\psi_{t}$ is absolutely continuous with respect to the Haar state with density $\rho_{t}$. Moreover, if it converges in $L^{2}(O_{N}^{+})$, then it converges also in $L^{1}(O_{N}^{+})$ (as the $L^{1}$-norm is dominated by the $L^{2}$-norm), and
\begin{equation*}
\|\rho_{t}\|_{2}^{2} = \sum_{n=0}^{+\infty}d_{n}^{2}e^{-2t\lambda_n}.
\end{equation*}
Recall that $d_{n} \leqslant q(N)^{-n}/(1-q(N)^{2})$. Using Lemma \ref{lem:affineestimate},
\begin{align*}
\ln\pg d_{n}^{2}e^{-2t\lambda_n}\pd & \leqslant 2n \ln(1/q(N)) -\ln(1-q(N)^2) -2t\lambda_n  \\
& \leqslant 2n \left(\ln(1/q(N)) -ta_{N}\right) + \ln(1/q(N)) -\ln(1-q(N)^2) - (b_{N} + c_{N})t \\
& \leqslant 2n \left(\ln(1/q(N)) -ta_{N}\right) + g(t, N)
\end{align*}
where $g(t, N)$ is a quantity independent of $n$. Consequently, $\|\rho_{t}\|_{2}$ is finite (so that $\psi_{t}$ is absolutely continuous with respect to the Haar state) as soon as
\begin{equation*}
\ln(1/q(N)) - ta_{N} < 0.
\end{equation*}

As for the second point, observe that
\begin{equation*}
\frac{\psi_{t}(\chi_{n})}{\|\chi_{n}\|_{\infty}} =  \frac{d_{n} e^{-t\lambda_n}}{n+1}.
\end{equation*}
If the right-hand side is not uniformly bounded with respect to $n$, then $\psi_{t}$ cannot extend to $L^{\infty}(O_{N}^{+})$. Using the previous lemma and the fact (see for instance \cite[Lem 3.8]{freslon2017cutoff}) that $d_n\geqslant Nq(N)^{-(n-1)}$, and proceeding as above, we see that $\frac{d_n e^{-t\lambda_n}}{n+1}$ will not be bounded as soon as
\begin{equation*}
\ln(1/q(N)) - ta_{N} > 0.
\end{equation*}
It follows from our two inequalities that $t_{abscont} = -\ln(q(N))a_{N}^{-1}$. Using the Taylor expansion of $q(N)^{-1}$ computed in the proof of \ref{prop:dlpn}, we see that $\ln(1/q(N)) = \ln(N) + O\pg \frac{1}{N^2}\pd$ so that
\begin{align*}
t_{abscont}(N) & = \ln(N) + O\left(\frac{1}{N^{2}}\right)\sqrt{N^{2}-4} \\
& = \left(N\ln(N) + O\left(\frac{1}{N}\right)\right)\left(1 - \frac{2}{N^{2}} + O\left(\frac{1}{N^{4}}\right)\right) \\
& = N\ln(N) - \frac{2\ln(N)}{N} + O\left(\frac{1}{N}\right).
\end{align*}
This concludes the proof.
\end{proof}

\begin{rem}
This is in sharp contrast with the classical case, where any non-degenerate (a condition analogous to requiring $b > 0$) Lévy process automatically has an $L^{2}$-density with respect to the Haar measure by \cite[Thm 1]{liao2004levy}, and is thus absolutely continuous.
\end{rem}

\begin{rem}
As we will see in Proposition \ref{thm:negativeprofile}, $\psi_{t}$ is absolutely continuous with respect to the Haar state if and only if the measure of the corresponding classical process is absolutely continuous with respect to the semi-circle distribution. Moreover, the lack of absolute continuity is witnessed by the appearance of an atom in that measure.
\end{rem}

In particular, $\psi_{N\ln(N)+cN}$ is absolutely continuous if and only if $c \geqslant -2\ln(N)/N^{2} + O\pg 1/N^2\pd$, which goes to $0$ as $N$ goes to infinity. Thus, the total variation distance is asymptotically only defined for $c \geqslant 0$.

\subsection{Cutoff profile}\label{subsec:limit_profile}

We will now try to get a better understanding of the cutoff phenomenon by computing the corresponding \emph{cutoff profile}, that is to say the limit of the distance between the process at time $t_c = N\ln(N) + cN$ and the Haar state as $N$ goes to infinity, $c$ being fixed. Our main result is an expression of this limit as the distance between two explicit probability measures. Before stating it, let us give some heuristics.

In the proof of Theorem \ref{thm:orthogonalcutoff}, we saw that it was enough to consider the element $\chi_{1}$ to obtain a lower bound of the correct order for the mixing time. In the case of a classical compact matrix group, $\chi_{1}$ is nothing but the trace function, and this would mean that the trace of the matrices is the last thing to be mixed by Brownian motion. In the case of $O_{N}^{+}$, we know that the distribution of $\chi_{1}$ under the Haar state is the semi-circle distribution $\nu_{\mathrm{SC}}$, so that we may expect the cutoff profile to be given by the distance between $\nu_{\mathrm{SC}}$ and a ``deformation'' of it. The whole problem of course lies in the vague meaning of the word ``deformation''.

We will show in the first part of Theorem \ref{thm:completeprofileorthogonal} that the profile indeed appears as the distance between $\nu_{\mathrm{SC}}$ and a family of closely related laws called the \emph{free Poisson distributions}. Let us recall that the free Poisson distribution with rate $\lambda$ and jump size $\alpha$ is given, for $\lambda > 1$, by (see \cite[Def 12.12]{nica2006lectures} for details)
\begin{equation*}
\mathrm{d}\FPoiss(\lambda, \alpha)(t) = \frac{1}{2\pi\alpha t}\sqrt{4\lambda\alpha^{2} - (t-\alpha(1+\lambda))^{2}}\mathrm{d}t.
\end{equation*}
Unfortunately, there is no value of the parameters for which the free Poisson distribution equals the semi-circle one.

One can nevertheless write things differently using a larger family of probability distribution called the \emph{free Meixner distributions}. Let us denote by $\FMeix(a, b)$ the standardised (i.e. with mean $0$ and variance $1$) free Meixner law with parameters $a$ and $b$ (see for instance \cite[Sec 2.2]{bozejko2006class} for details). Its absolutely continuous part with respect to the Lebesgue measure is given by
\begin{equation*}
\mathrm{d}\FMeix(a, b)(t) = \frac{\sqrt{4(1+b) - (t-a)^{2}}}{2\pi(bt^{2} + at + 1)}\mathbf{1}_{[a-2\sqrt{1+b}, a+2\sqrt{1+b}]}\mathrm{d}t.
\end{equation*}
For $a = b = 0$, the formula reduces to the density of the semi-circular distribution, while for $b = 0$ it yields the density of a free Poisson distribution with mean $0$ and variance $1$.

We will now state our result using both the free Poisson and the free Meixner settings, after introducing some extra notations. If $X$ is a random variable with law $\mu$, then we denote by $D_{r}(\mu)$ the $r$-dilation of $\mu$ (that is to say the law of $rX$) and by $\mu\ast\delta_{a}$ its translation by $a$ (that is to say the law of $X+a$). Moreover, we denote by $\dvt$ the usual total variation distance for Borel measure on $\R$ and by $\widetilde{\mathrm{d}}_{N}$ the distance associated to the norm $\|\cdot\|$ on $O_{N}^{+}$.

\begin{thm}\label{thm:completeprofileorthogonal}
Let $c\in \R$, and recall $t_{c} = N\ln(N) + cN$. Then
\begin{align*} 
\widetilde{\mathrm{d}}_{N}\pg\psi_{t_{c}}, h\pd \xrightarrow[N\to \infty]{} f(c)
:= & \dvt\pg\FPoiss\pg e^{2c}, -e^{-c}\pd\ast\delta_{e^{c} + e^{-c}}, \nu_{\mathrm{SC}}\pd \\ 
= & \dvt\pg\FMeix\pg-e^{-c}, 0\pd\ast\delta_{e^{-c}}, \FMeix(0, 0)\pd.
\end{align*}
\end{thm}

As the proof of this result is long, we will split it in two, depending on the sign of $c$. For convenience and clarity, the two cases will be stated in the separate Propositions \ref{prop:rightprofileorthogonal} and \ref{thm:negativeprofile}, each treated in a proper subsection.

\subsubsection{The profile on the right}\label{subsec:rightprofile}

We start with the case $c > 0$, which turns out to be the simplest one. The reason for this is that it follows from Proposition \ref{prop:absolutecontinuity} that $\psi_{t_{c}}$ is absolutely continuous with respect to the Haar state in that case, so that the convergence in total variation distance boils down to $L^{1}$-convergence of the densities, which in turn follows from $L^{2}$-convergence. The main part of the work is therefore rather the identification of the limit and its expression in terms for of Poisson or free Meixner laws.

\begin{prop}\label{prop:rightprofileorthogonal}
For $c > 0$,
\begin{align*} 
\widetilde{\mathrm{d}}_{N}\pg\psi_{t_{c}}, h\pd \xrightarrow[N\to \infty]{} f(c)
= & \dvt\pg\FPoiss\pg e^{2c}, -e^{-c}\pd\ast\delta_{e^{c} + e^{-c}}, \nu_{\mathrm{SC}}\pd \\ 
= & \dvt\pg\FMeix\pg-e^{-c}, 0\pd\ast\delta_{e^{-c}}, \FMeix(0, 0)\pd.
\end{align*}
\end{prop}

\begin{proof}
Recall that as $c > 0$, $\psi_{t_{c}}$ has an $L^{1}$-density given by 
\begin{equation*}
\rho_{t_{c}} = \sum_{n=0}^{+\infty}d_{n}e^{-t_{c}\lambda_n}\chi_{n}.
\end{equation*}
Moreover, we know from Lemma \ref{lem:affineestimate} that
\begin{equation*}
\lambda_n = n\pg\frac{1}{N} + \underset{N\to \infty}{O} \left(\frac{1}{N^3}\right)\pd,
\end{equation*}
and an easy computation yields $d_{n}\underset{N\to \infty}{\sim} N^{n}$. In particular, for each $n$, $d_{n}e^{-t_{c}\lambda_n}$ converges to $e^{-cn}$ as $N$ goes to $+\infty$. Moreover,
\begin{equation*}
\|d_{n}e^{-t_{c}\lambda_n}\chi_{n}\|_{1}\leqslant \|d_{n}e^{-t_ {c}\lambda_n}\chi_{n}\|_{2} = d_{n}e^{-t_{c}\lambda_n}
\end{equation*}
and because $qN\geqslant 1$, for $N \geqslant 3$,
\begin{equation*}
d_{n}e^{-t_{c}\lambda_n} \leqslant \frac{q^{-n}}{(1-q^{2})}N^{-n}e^{-nc} \leqslant \frac{3}{2}e^{-nc}.
\end{equation*}
The latter being summable and independent of $N$, we can exchange the sum over $n$ and the limit in $N$. This yields (recall that there is an isomorphism between $\O(O_{N}^{+})_{\text{central}}$ and $\C[X]$ sending $\chi_{n}$ to $P_{n}$ and sending the measure associated to $h$ to the semi-circle distribution)
\begin{equation*}
\lim_{N\to +\infty}\|\psi_{t_{c}} - h\| = \frac{1}{2}\left\|\sum_{n=1}^{+\infty}e^{-cn}P_{n}\right\|_{1},
\end{equation*}
where the $L^{1}$-norm is computed with respect to the standard semi-circular distribution. Using the generating series of the Chebyshev polynomials of the second kind (which is easily computed, multiplying by $t$ and using the recursion relation), we get for every $t \in \cg -2,2\cd$,
\begin{align*}
\sum_{n=1}^{+\infty}e^{-cn}P_{n}(t) & = \frac{1}{1 - te^{-c} + e^{-2c}} - 1 \\
& = \frac{1}{1+\beta^{2}}\frac{1}{1-\gamma t} - 1 \\
& = F_{c}(t) - 1.
\end{align*}
where $\beta = e^{-c}$ and $\gamma = \beta/(1+\beta^{2}) < 1/2$. Thus, the cutoff profile is equal to
\begin{equation*}
\lim_{N\to +\infty}\|\psi_{t_{c}} - h\| = \frac{1}{2}\int_{-2}^{2}\left\vert F_{c}(t) - 1\right\vert\mathrm{d}\nu_{\mathrm{SC}}(t).
\end{equation*}
Performing the change of variables $u = 1-\gamma t$,
\begin{align*}
F_c(t)\mathrm{d}\nu_{\mathrm{SC}}(t) = F_{c}(t)\frac{\sqrt{4-t^{2}}}{2\pi}\mathbf{1}_{[-2, 2]}(t)\mathrm{d}t & = \frac{1}{2\pi(1+\beta^{2})u}\sqrt{4 - \left(\frac{1-u}{\gamma}\right)^{2}}\mathbf{1}_{[1-2\gamma, 1+2\gamma]}(u)\frac{\mathrm{d}u}{\gamma} \\
& = \frac{1}{2\pi\gamma^{2}(1+\beta^{2})u}\sqrt{4\gamma^{2} - \left(1-u\right)^{2}}\mathbf{1}_{[1-2\gamma, 1+2\gamma]}(u)\mathrm{d}u.
\end{align*}
Setting $\alpha = \beta\gamma = \gamma^{2}(1+\beta^{2})$ and $\lambda = \beta^{-2} > 1$, this density becomes
\begin{equation*}
\frac{1}{2\pi\alpha u}\sqrt{4\lambda\alpha^{2} - (u - \alpha(1+\lambda))^{2}}\mathbf{1}_{[\alpha(1-\sqrt{\lambda})^{2}, \alpha(1+\sqrt{\lambda})^{2}]}(u)\mathrm{d}u.
\end{equation*}
This is exactly the free Poisson distribution with rate $\lambda = e^{2c}$ and jump size $\alpha = e^{-2c}/(1+e^{-2c})$. Reversing the change of variables, we see that $F_{c}(t)\mathrm{d}\nu_{\mathrm{SC}}(t)$ is the density of the law
\begin{equation*}
D_{-1/\gamma}\left(\FPoiss\pg\beta^{-2}, -\beta\gamma\pd\ast\delta_{-1}\right) = \FPoiss\pg\beta^{-2}, -\beta\pd\ast\delta_{1/\gamma} = \FPoiss\pg e^{2c}, -e^{-c}\pd\ast\delta_{e^{c} + e^{-c}},
\end{equation*}
hence the result. Using the facts that $\FPoiss\pg a^{-2}, a\pd\ast\delta_{-a^{-1}} = \FMeix(a, 0)$ and that $\nu_{\mathrm{SC}} = \FMeix(0, 0)$, the second formula follows.
\end{proof}

As explained heuristically at the beginning of this subsection, the fact that Brownian motion is not completely mixed is witnessed by the ``trace'' it can attain, and the cutoff profile gives a precise quantitative description of this phenomenon. In particular, it shows that the ``trace'' of Brownian motion is averagely shifted to the right and more concentrated around its mean. Here is a plot of the density of $\FMeix(-e^{-c}, 0)\ast\delta_{e^{-c}}$ with respect to the Lebesgue measure for values of $c$ between $0$ and $5$ :
\begin{center}
\includegraphics[scale=2]{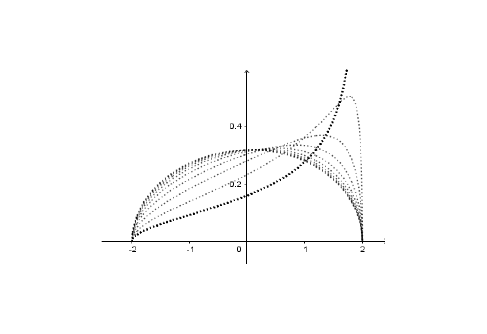}
\end{center}
For $c = 0$ we get a free Poisson law (this is the curve with a peak on the right) while for $c = 5$ the density is already indistinguishable to the naked eye from that of the semi-circle distribution.

\begin{rem}\label{rem:explicitprofile}
The quantity $\dvt\pg\FMeix\pg-e^{-c}, 0\pd\ast\delta_{e^{-c}}, \FMeix(0, 0)\pd$ can be computed explicitly in terms of $c$ through integration, yielding the formula :
\begin{equation*}
f(c) = \frac{e^{2c}-1}{2\pi}\arcsin\left(\frac{e^{-3c}-3e^{-c}}{2}\right) + \frac{1-e^{2c}}{2\pi}\arcsin\left(\frac{e^{-c}}{2}\right) + \frac{e^{-2c}+2}{4 \pi}\sqrt{4e^{2c}-1}.
\end{equation*}
Using the facts that $0 < \frac{e^{-c}}{2} < 1/2$ and that $3\arcsin(x) = \arcsin(3x - 4x^{3})$ for $0\leqslant x \leqslant 1/2$, this can be rewritten as
\begin{equation*}
f(c) = (1 - e^{2c})\frac{2}{\pi}\arcsin\left(\frac{e^{-c}}{2}\right) +  \frac{e^{-2c}+2}{4 \pi}\sqrt{4e^{2c}-1}.
\end{equation*}
\end{rem}

\subsubsection{The profile on the left}

By Proposition \ref{prop:absolutecontinuity}, the computations above can only make sense for $c \geqslant 0$. However, the free Poisson distribution makes sense even for $\lambda < 1$, with the only difference that some mass is carried by an atom :
\begin{equation*}
\mathrm{d}\FPoiss\pg\lambda, \alpha \pd (t) = \left(1-\lambda\right)\delta_{0} + \frac{\lambda}{2\pi\alpha t}\sqrt{4\lambda\alpha^{2} - (t-\alpha(1+\lambda))^{2}}\mathbf{1}_{[\alpha(1-\sqrt{\lambda})^{2}, \alpha(1+\sqrt{\lambda})^{2}]}(t)\mathrm{d}t
\end{equation*}

As a consequence, the formula
\begin{equation*}
f(c) = \dvt\pg\FPoiss\pg e^{2c}, -e^{-c}\pd\ast\delta_{e^{c} + e^{-c}}, \nu_{\mathrm{SC}}\pd
\end{equation*}
does indeed make sense for $c < 0$ and we will now show that this is indeed the profile. Let us start with a characterization of the limit distribution in terms of ``Chebyshev moments''.

\begin{lem}\label{lem:chebyshevmoments}
For any $c\in \R$, the measure
\begin{equation*}
\mu_{c} = \FPoiss\pg e^{2c},-e^{-c}\pd\ast\delta_{e^{c}+e^{-c}}
\end{equation*}
is the unique probability measure on $\R$ such that for any $n\in \N$,
\begin{equation*}
\int_{\R}P_{n}\mathrm{d}\mu_{c} = e^{-cn}.
\end{equation*}
\end{lem}

\begin{proof}
Let us first recall that the free cumulants (see \cite[Def 11.3]{nica2006lectures} for the definition of free cumulants and \cite[Prop 12.11]{nica2006lectures} for the free Poisson case) of the free Poisson distribution $\FPoiss\pg\lambda, \alpha \pd$ are given by
\begin{equation*}
\kappa_{n} = \lambda\alpha^{n}.
\end{equation*}
As a consequence, the free cumulants of $\FPoiss\pg e^{2c},-e^{-c}\pd$ are Laurent polynomials in $e^{c}$. The free additive convolution with $\delta_{e^{c}+e^{-c}}$ only modifies the first cumulant $\kappa_{1}$ by adding $e^{c}+e^{-c}$ to it, hence the free cumulants of $\mu_{c}$ are Laurent polynomials in $e^{c}$.
Because the moments are polynomial functions of the free cumulants (by virtue of the moment-cumulant formula, see \cite[Prop 11.4]{nica2006lectures}), we conclude that there exist Laurent polynomials $L_{n}$ 
such that for any $c\in \R$,
\begin{equation*}
\int_{\R}P_{n}(x)\mathrm{d}\mu_{c}(x) = L_{n}(e^{c}).
\end{equation*}
Let us now assume that $c > 0$. Then, we know from the proof of Theorem \ref{thm:completeprofileorthogonal} in Subsection \ref{subsec:rightprofile} that for $c>0$, $\mu_{c}$ is absolutely continuous with respect to the semi-circle distribution $\nu_{\text{SC}}$ with density $F_{c}$. Using the fact that the polynomials $P_{n}$ are orthonormal for the semi-circle distribution, we get
\begin{equation*}
\int_{\R}P_{n}\mathrm{d}\mu_{c} = \int_{\R}P_{n}\left(\sum_{n=0}^{+\infty}e^{-cn}P_{n}\right)\mathrm{d}\nu_{\text{SC}} = e^{-cn}.
\end{equation*}
As a consequence, $L_{n}(x) = x^{-n}$ for any $x > 1$, so that by uniqueness of the decomposition of a Laurent polynomial, $L_{n}(X) = X^{-n}$ and the formula for the integral of Chebyshev polynomials also holds for $c\leqslant 0$.

As for the uniqueness assertion, it simply follows from the fact that $\mu_{c}$ has compact support and is consequently determined by its moments.
\end{proof}

With this in hand, we can complete our proof of the cutoff profile. To do so, let us define, for $t > 0$ and $N\in \N$, a probability measure $m_{t}^{(N)}$ on $\R$ by requiring that $m_{t}^{(N)}(\R\setminus[-N, N]) = 0$ and for all $n\in \N$,
\begin{equation*}
\int_{-N}^{N}P_{n}(x)\mathrm{d}m_{t}^{(N)} = e^{-t\lambda_n}P_{n}(N).
\end{equation*}
Such a measure exists because it corresponds to the restriction of the state $\psi_{t}$ to the algebra $\O(O_{N}^{+})_{\text{central}}$. Using the centrality of the process and the Riesz representation theorem (see the discussion at the beginning of Section \ref{subsec:cutofforthogonal}), we know that
\begin{equation*}
\left\|\psi_{t_c} - h\right\| = \left\|m_{t_{c}}^{(N)} - \nu_{\text{SC}}\right\|_{TV}.
\end{equation*}
To find the limit of the right-hand side, we will compute $m_{t_{c}}^{(N)}$ explicitly. Our approach is based on an idea of P. Biane in \cite[Sec 12.2]{biane2008introduction} for the study of a similar process on the duals of free groups. One would naively want the measure to be absolutely continuous with respect to the Haar measure with density
\begin{equation*}
\sum_{n=0}^{+\infty} e^{-t\lambda_n}P_{n}(N)P_{n}
\end{equation*}
However, we know that for $t = N\ln(N) + cN$ with $c < 0$, this series does not converge in $L^{2}([-2, 2])$ because $e^{-t\lambda_n}P_{n}(N)$ grows exponentially (this is the lack of absolute continuity with respect to the Haar state). The idea is to find a point $\widetilde{N}(t)$ such that $e^{-t\lambda_n}P_{n}(N) - P_{n}(\widetilde{N}(t))$ is small enough to be summable. Then summing the previous difference to produce an absolutely continuous measure and adding a Dirac mass at $\widetilde{N}(t)$ will produce an expression of the measure. To find that $\widetilde{N}(t)$, let us first recall that for $x > 2$, if $0 < q(x) < 1$ denotes the unique number such that $q(x) + q(x)^{-1} = x$, then
\begin{equation*}
P_{n}(x) = \frac{q(x)^{-(n+1)} - q(x)^{n+1}}{q(x)^{-1} - q(x)}.
\end{equation*}
In particular, considering (with the notations of Lemma \ref{lem:affineestimate}) that $\lambda_{n}$ is roughly equal to ${a_{N}}n$, we have
\begin{align*}
e^{-t\lambda_{n}}P_{n}(N) & \approx e^{-nt{a_{N}}}\frac{q(N)^{-(n+1)} - q(N)^{n+1}}{q(N)^{-1} - q(N)} \\
& \approx e^{-nt{a_{N}}}q(N)^{-n} \\
& \approx P_{n}\left(e^{-{a_{N}}t}q(N)^{-1} + e^{{a_{N}}t}q(N)\right)
\end{align*}
so that setting
\begin{equation*}
\widetilde{N}(t) = e^{-{a_{N}}t}q(N)^{-1} + e^{{a_{N}}t}q(N)
\end{equation*}
has a good chance of giving the correct order of magnitude. There is still a normalization coefficient required, and the next result is meant to make all this precise and rigorous. We will write $q$ for $q(N)$ in order to lighten notations.

\begin{lem}\label{lem:absolutepartnegativec}
For any fixed $c < 0$ and $t_{c} = N\ln(N) + cN$, there exists $N_{c}$ such that for all $N\geqslant N_{c}$,
\begin{equation*}
m_{t_{c}}^{(N)} = \alpha(t_{c})\delta_{\widetilde{N}(t_{c})} + \sum_{n=0}^{+\infty}\pg e^{-\lambda_n t_{c}}P_{n}(N) - P_{n}(\widetilde{N}(t_{c})\pd P_{n}(y)\mathrm{d}\mu_{\text{SC}}(y),
\end{equation*}
where
\begin{equation*}
\alpha(t) = e^{{a_N}t - {b_N}t}\frac{e^{-{a_N}t}q^{-1} - e^{{a_N}t}q}{q^{-1} - q} 
\end{equation*}
and 
\begin{equation*}
\sum_{n\geqslant 0} \left(e^{-t_{c}\lambda_n}P_{n}(N) - \alpha(t_{c})P_{n}(\widetilde{N}(t_{c}))\right)P_{n} \in L^{2}([-2, 2], \nu_{\text{SC}}).
\end{equation*}
\end{lem}

\begin{proof}
To lighten notations, let us simply write $a, b, c$ for $a_{N}, b_{N}, c_{N}$. Since by definition $q\pg\widetilde{N}(t)\pd = (e^{{a}t}q)^{\pm 1}$ (depending on which one is less than one, but this does not change the expression of $P_{n}(\widetilde{N}(t))$),
\begin{align*}
\pg q^{-1}-q \pd \alpha(t)P_{n}\pg\widetilde{N}(t)\pd & = e^{at - bt} \pg e^{-{a}t}q^{-1} - e^{{a}t}q \pd\frac{q\pg\widetilde{N}(t)\pd^{-n-1} - q\pg\widetilde{N}(t)\pd^{n+1}}{q\pg\widetilde{N}(t)\pd^{-1} - q\pg\widetilde{N}(t)\pd} \\
& = e^{{a}t - {b}t}\pg e^{-{a}(n+1)t}q^{-(n+1)} - e^{{a}(n+1)t}q^{n+1}\pd \\
& = e^{-({a}n + {b})t}q^{-(n+1)} - e^{{a}(n+2)t - {b}t}q^{n+1}, \\
\end{align*}
from which we deduce that
\begin{align*}
    \alpha(t)P_{n}\pg\widetilde{N}(t)\pd  = e^{-({a}n + {b})t}P_n(N) +\frac{q^{n+1}}{q^{-1}-q}\pg e^{-({a}n + {b})t} - e^{(a(n+2)-b)t} \pd,
\end{align*}
which eventually leads to the equality
\begin{align*}
 \alpha(t)P_{n}\pg\widetilde{N}(t)\pd - e^{-\lambda_{n}t}P_n(N) =\pg e^{{c(n)}t} - 1 \pd e^{-\lambda_{n}t}P_{n}(N) + \frac{q^{n+1}}{q^{-1}-q}\pg e^{-(an + b)t} - e^{(a(n+2)-b)t} \pd.
\end{align*}
Using that $e^y-1\leqslant (e-1)y \leqslant 2y$ for $0\leqslant y \leqslant 1$ and using Lemma \ref{lem:affineestimate}, we see that for $N\geqslant 4$,
\begin{equation*}
0 \leqslant e^{{c(n)}t} - 1 \leqslant 2\pg\frac{2}{N}\pd^{n}.
\end{equation*}
Moreover, $a \geqslant 1/N$ and $b, c(n) \geqslant 0$ so that $e^{-\lambda_{n} t} \leqslant e^{-ant} \leqslant e^{-nt/N}$ (remember that $c < 0$). Taking $t = t_{c} = N\ln(N) + cN$ (and of course $N$ large enough so that $t_{c} > 0$), and using also  $\abs{P_{n}(N)} \leqslant N^{n}$ we can bound the first term :
\begin{equation*}
\abs{\pg e^{{c(n)}t_{c}} - 1 \pd e^{-\lambda_{n} t_{c}}P_{n}(N)} \leqslant 2\pg\frac{2}{N}\pd^{n} e^{-nt_{c}/N} N^{n} = 2 \pg\frac{2e^{-c}}{N}\pd^{n}.
\end{equation*}
Taking $N$ greater than some $N_{c}$ depending on $c$, this converges exponentially fast to $0$.

Let us now bound the second term. First observe that $q^{-1}-q = \sqrt{N^{2} - 4}$ and $e^{-(an + b)t}\leqslant 1 \leqslant e^{((n+2)-b)t}$, hence we have
\begin{equation*}
\abs{\frac{q^{n+1}}{q^{-1}-q}\pg e^{-(an + b)t} - e^{(a(n+2)-b)t} \pd} \leqslant \frac{q^{n+1}}{\sqrt{N^{2}-4}} e^{(a(n+2)-b)t} \leqslant \frac{q^{n+1}}{\sqrt{N^{2}-4}} e^{a(n+2)t}.
\end{equation*}
At $t = t_{c}$, we then get as $N\to \infty$, using $a = \frac{1}{N} + O\pg \frac{1}{N^{3}}\pd$, that
\begin{align*}
\frac{q^{n+1}}{\sqrt{N^2-4}} e^{a(n+2)t_{c}} & = \pg \frac{1}{N} + O\pg \frac{1}{N^{3}}\pd\pd^{n+2} \pg e^{a(N\ln(N) + cN)}\pd^{n+2} \\
& = \pg e^{-\ln(N) + O(1/N^{2})}\pd^{n+2} \pg e^{\ln(N) + c + O(\ln(N)/N^2)}\pd^{n+2} \\
& = \pg e^{c+o(1)}\pd^{n+2}.
\end{align*}
Hence, as $c < 0$ and $N\geqslant N_{c}$ is fixed, this term converges exponentially fast to $0$ as $n\to\infty$. Combining those two bounds, we conclude that for $N$ fixed large enough,
\begin{equation*}
\abs{\alpha(t)P_{n}\pg\widetilde{N}(t)\pd - e^{-\lambda_n t}P_n(N)}
\end{equation*}
converges exponentially fast to $0$ as $n\to\infty$. Because $(P_{n})_{n\in \N}$ is a Hilbert basis for $L^{2}([-2, 2], \nu_{\text{SC}})$, this proves that the sum stated in the lemma belongs to $L^{2}([-2, 2], \nu_{\text{SC}})$.

Now the first formula in the statement makes sense for $N$ large enough and defines a probability measure on $[-N, N]$. Moreover, the integral of any Chebyshev polynomial with respect to that measure coincides by construction with its integral with respect to $m_{t_{c}}^{(N)}$. Because Chebyshev polynomials form a basis of $\C[X]$, it follows that the two measures have the same moments and hence coincide since they have compact support.
\end{proof}

We are now ready to establish the cutoff profile for negative $c$.

\begin{prop}\label{thm:negativeprofile}
For any fixed $c < 0$ and $t_{c} = N\ln(N) + cN$, we have
\begin{equation*}
\widetilde{\mathrm{d}}_{N}(\psi_{t_{c}}, h) = \left\|m_{t_{c}}^{(N)} - \nu_{\text{SC}}\right\|_{TV}\underset{N\to+\infty}{\longrightarrow} \left\|\mathrm{Poiss}^{+}\pg e^{2c}, -e^{-c}\pd\ast\delta_{e^{c} + e^{-c}}, \nu_{\mathrm{SC}}\right\|_{TV}.
\end{equation*}
\end{prop}

\begin{proof}
To compute the limit of the total variation distance, let us first notice that if $\widetilde{m}_{t_{c}}^{(N)} = \mathbf{1}_{[-2, 2]}m_{t_{c}}^{(N)}$, then
\begin{align*}
\left\|m_{t_{c}}^{(N)} - \nu_{\text{SC}}\right\|_{TV} & = \left\|\mathbf{1}_{[-2, 2]}\left(m_{t_{c}}^{(N)} - \nu_{\text{SC}}\right)\right\|_{TV} + \left\|\mathbf{1}_{\R\setminus[-2, 2]}\left(m_{t_{c}}^{(N)} - \nu_{\text{SC}}\right)\right\|_{TV} \\
& = \left\|\widetilde{m}_{t_{c}}^{(N)} - \nu_{\text{SC}}\right\|_{TV} + \alpha(t_{c}).
\end{align*}
We will deal with each part separately :
\begin{itemize}
\item It is straightforward to see that as $N$ goes to infinity, $\alpha(t_{c})\to 1-e^{2c}$ and that $\widetilde{N}(t_{c})\to e^{c} + e^{-c}$, which are exactly the parameters of the atom of the free Poisson distribution in the statement.
\item Observe moreover that $m_{t_{c}}^{(N)}$ converges in moments to a measure for which the integral of $P_{n}$ equals $e^{cn}$. By Lemma \ref{lem:chebyshevmoments}, this must be the free Poisson distribution in the statement.
Therefore, if we can prove that $\widetilde{m}_{t_{c}}^{(N)}$ converges in total variation distance, then its limit must be the absolutely continuous part of $\mathrm{Poiss}^{+}\pg e^{2c}, -e^{-c}\pd \ast\delta_{e^{c} + e^{-c}}$.
The density of $\widetilde{m}_{t_{c}}^{(N)}$ with respect to $\nu_{\text{SC}}$ is
\begin{equation*}
\sum_{n=0}^{+\infty}\pg e^{-\lambda_n t_{c}}P_{n}(N) - P_{n}(\widetilde{N}(t))\pd P_{n}
\end{equation*}
and each term of the sum converges as $N$ goes to infinity. Therefore, if we can bound these terms in $L^{2}$-norm by a summable sequence not depending on $N$, then we can conclude. But this was already done in Lemma \ref{lem:absolutepartnegativec}, since it follows from it that such a bound exists at least for $N$ large enough depending on $c$.
\end{itemize}
\end{proof}

\begin{rem}
Note that we do not prove that $m_{t_{c}}^{(N)}$ converges in total variation to $\mathrm{Poiss}^{+}\pg e^{2c}, -e^{-c}\pd\ast\delta_{e^{c} + e^{-c}}$. This is actually false since their atoms are not at the same place. If we change slightly the time by setting
\begin{equation*}
\widehat{t}_{c} = \frac{-\ln(q(N))+ c}{a_{N}} = t_{abscont}(N) + \frac{c}{a_{N}},
\end{equation*}
then $\Tilde{N}(\widehat{t}_{c}) = e^{c} + e^{-c}$ and $m_{\widehat{t}_{c}}^{(N)}$ converges in total variation distance to $\mathrm{Poiss}^{+}\pg e^{2c}, -e^{-c}\pd\ast\delta_{e^{c} + e^{-c}}$. However, this would be somewhat unnatural, and we prefer to state only results which are true for all times $t_c + o(N)$ than for a very specific sequence of times.
\end{rem}

\begin{rem}
Once again, the integral giving the total variation distance can be computed explicitly in terms of $c$. The computation differs however depending on whether $c$ is greater or smaller than $-\ln(2)$ (because of the absolute value in the integral). For $-\ln(2) < c < 0$, the result is the same as the first expression given in Remark \ref{rem:explicitprofile} for $c > 0$, except for an extra $(1-e^{2c})/2$. However, this time $1/2\leqslant e^{-c}/2\leqslant 1$ and for $x\in [1/2, 1]$, $\pi - 3\arcsin(x) = \arcsin(3x - 4x^{3})$ so that we end up with the same result, namely
\begin{equation*}
f(c) = (1-e^{2c})\frac{2}{\pi}\arcsin\left(\frac{e^{-c}}{2}\right) + \frac{e^{-2c}+2}{4 \pi e^{-c}}\sqrt{4-e^{-2c}}
\end{equation*}
while for $c\leqslant -\ln(2)$,
\begin{equation*}
f(c) = 1 - e^{2c}.
\end{equation*}
\end{rem}

The combination of Propositions \ref{prop:rightprofileorthogonal} and \ref{thm:negativeprofile} yields the complete proof of Theorem \ref{thm:completeprofileorthogonal}, therefore yielding the cutoff profile for Brownian motion on $O_{N}^{+}$. As an illustration, here is a plot of the profile :
\begin{center}
\includegraphics[scale=0.25]{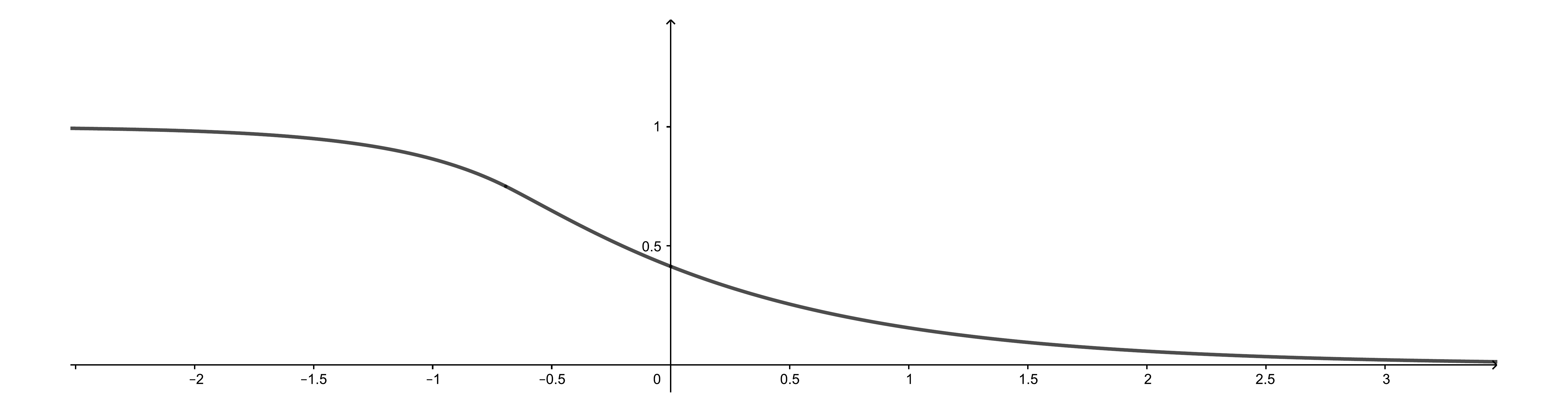}
\end{center}

\begin{rem}\label{rem:profileimpliescutoff}
Theorem \ref{thm:completeprofileorthogonal} is not just a refinement of the cutoff phenomenon established in Theorem \ref{thm:orthogonalcutoff}, but a stronger result. In other words, the existence of a cutoff profile implies the cutoff phenomenon, as was already mentioned in the beginning of the proof of Theorem \ref{thm:orthogonalcutoff}. This follows from the fact that the profile converges to $1$ at $-\infty$ and $0$ at $+\infty$, together with the monotonicity of Proposition \ref{lem:decreasing}.
\end{rem}

\subsection{Further results}

Let us complete this section with some additional remarks and results concerning various generalizations of the original problem.

\subsubsection{Other norms}

We have worked so far with the Fourier-Stieltjes norm, because it is the only natural norm available which makes sense for all $t\in \R_{+}$. However, the upper bound was computed using the total variation distance, and one may wonder whether the cutoff upper bound also holds with respect to other distances. It turns out that the answer is yes.

\begin{cor}
Brownian motion on $O_{N}^{+}$ satisfies, for all $1\leqslant p\leqslant \infty$ and $c > 0$, with $t_{c} = N\ln(N) + cN$,
\begin{equation*}
\lim_{N\to +\infty}\|\psi_{t_c} - h\|_{L^{p}} = \left\|\FMeix\pg-e^{-c}, 0\pd\ast\delta_{e^{-c}} - \FMeix(0, 0)\right\|_{L^{p}}.
\end{equation*}
\end{cor}

\begin{proof}
Recall that the density of the process at time $t$ is, if the series makes sense,
\begin{equation*}
\rho_{t} = \sum_{n=0}^{+\infty}d_{n}e^{-t\lambda_n }\chi_{n}.
\end{equation*}
Using the fact that $\|\chi_{n}\|_{\infty} = P_{n}(2) = n+1$, we see that the density converges in $L^{\infty}$-norm at $t_{c}$ as soon as $c > 0$ since (for $N \geqslant 3$) and
\begin{equation*}
\left\|d_{n}e^{-t_{c}\lambda_n}\chi_{n}\right\|_{\infty}\leqslant (n+1)d_{n}e^{-t_{c}\lambda_n } \leqslant \frac{3}{2}(n+1)e^{-nc}.
\end{equation*}
Moreover, using this bound the same strategy as for Proposition \ref{prop:rightprofileorthogonal} yields the cutoff profile in the case $p = \infty$. As for finite $p$, it follows from the noncommutative H\"{o}lder inequality (see for instance \cite[Thm 2.13.iv]{takesaki2003theoryII}) that for any $1\leqslant p\leqslant \infty$,
\begin{equation*}
\|d_{n}e^{-t_{c}\lambda_n}\chi_{n}\|_{p}\leqslant \|d_{n}e^{-t_{c}\lambda_n}\chi_{n}\|_{\infty}
\end{equation*}
hence we can once again resort to the same argument.
\end{proof}

\begin{rem}
Note that for $c<0$, according to the discussions in the previous subsection, we may decompose $\psi_{t_{c}}$ into an absolutely continuous part $\widetilde{\psi}_{t_{c}}$ which admits a $L^p$-density for all $1\leqslant p \leqslant \infty$ and such that $\psi _{t_c} -  \widetilde{\psi}_{t_c}$ is singular with respect to $h$. It is easy to see that the above theorem still holds for $\widetilde{\psi}_{t_{c}}$ even for $c<0$.
\end{rem}

Let us compare this with the classical case. P.-L. Méliot proved in \cite[Thm 7]{meliot2014cut}, building on results of G. Chen and L. Saloff-Coste in \cite{chen2008cutoff}, that the cutoff phenomenon for Brownian motion on $SO(N)$ indeed occurs for all $1\leqslant p\leqslant \infty$ and that the cutoff time is the same as for the $L^{1}$-norm for all $1\leqslant p < \infty$. However, for $p = \infty$, the cutoff time is doubled and becomes $4\ln(N)$. It is therefore quite surprising that in the quantum case, the difference between the case of finite and infinite parameter $p$ disappears. This must nevertheless be tempered with the fact that there is no lower bound in our case, because our definition of the $L^{p}$ norm only makes sense for absolutely continuous states.

\subsubsection{The free real sphere}

In \cite{meliot2014cut}, P.-L. Méliot did not only prove the cutoff phenomenon for compact simple Lie groups, but also for their homogeneous spaces. In the quantum setting, there is no structure theory of compact homogeneous spaces paralleling the classical one, but there are nevertheless some explicit examples. We will now consider the simplest of them, which is an analogue of the real sphere on which the classical orthogonal group acts. The same idea of ``liberation'' as for the definition of free orthogonal quantum groups suggests that the free analogue of the real sphere should be described by the universal $*$-algebra generated by $N$ self-adjoint elements $(x_{i})_{1\leqslant i\leqslant N}$ such that
\begin{equation*}
\sum_{i=1}^{N}x_{i}^{2} = 1.
\end{equation*}
Denoting by $\O(S_{+}^{N-1})$ this object, it is endowed with an action of $O_{N}^{+}$ through the map
\begin{equation*}
\alpha : x_{i}\mapsto \sum_{i=1}^{N}u_{ij}\otimes x_{j}.
\end{equation*}
Note that the abelianization of $\O(S_{+}^{N-1})$ is exactly the algebra of polynomial functions on the $N-1$ dimensional sphere in $\R^{N}$ and that the formula defining $\alpha$ also defines the usual action of $O_{N}$ on that sphere.

Intuitively, Brownian motion on such a space should be a Lévy process invariant under the action $\alpha$, and the analogue of the uniform measure should be the unique probability measure invariant under $\alpha$. Such an $\alpha$-invariant state does indeed exist and can be constructed in the following way.
Consider the subalgebra $\O(X_{N})\subset \O(O_{N}^{+})$ generated by the elements $u_{i1}$ for $1\leqslant i\leqslant N$. Then, there is a surjective $*$-homomorphism $\pi : \O(S_{+}^{N-1})\to \O(X_{N})$ sending $x_{i}$ to $u_{i1}$. Moreover, one has
\begin{equation*}
(\pi\otimes \id)\circ\alpha = (\D\otimes\id)\circ\pi
\end{equation*}
so that the state $\omega = h\circ\pi$ is invariant under the action $\alpha$. As a consequence, we will only consider the ``concrete'' model $\O(X_{N})$ instead of $\O(S_{+}^{N-1})$.

Brownian motion considered in \cite{meliot2014cut} on a homogeneous space is then the projection of Brownian motion coming from the group. In our case, this simply amounts to restricting $\psi_{t}$ to $\O(X_{N})$. Before giving the expression, let us first recall that by \cite[Lem 7.3]{das2018invariant}, one may find a basis for the carrier Hilbert space of each irreducible representation $u^{n}$ of $O_{N}^{+}$ such that
\begin{equation*}
\O(X_{N}) = \Span\{u_{i1}^{n} \mid 1\leqslant i\leqslant N, n\in \N\}.
\end{equation*}

\begin{prop}\label{prop:cutoff_sphere}
The Lévy process given by the restriction of $\psi_{t}$ to $\O(X_{N})$ exhibits a cutoff phenomenon at time $t_{N} = \frac{1}{2} N\ln(N)$. Morevoer, it has the same cutoff profile as Brownian motion on $O_{N}^{+}$.
\end{prop}

\begin{proof}
For $t$ large enough, the density of $\psi_{t} - h$ is
\begin{equation*}
\sum_{n=1}^{+\infty}d_{n}e^{-tP_{n}'(N)/P_{n}(N)}u_{11}^{n}
\end{equation*}
whose $L^{2}$-norm squared is
\begin{equation*}
\sum_{n=1}^{+\infty}d_{n}e^{-2tP_{n}'(N)/P_{n}(N)}.
\end{equation*}
The difference with the previous case is that the dimension $d_{n}$ is not squared, due to the fact that coefficients of irreducible representations form an orthogonal but not orthonormal basis. This accounts for the factor $1/2$ in the cutoff time, exactly as in \cite{meliot2014cut}. The proof is now exactly the same as for Theorems \ref{thm:orthogonalcutoff} and \ref{thm:completeprofileorthogonal}.
\end{proof}

There is also another candidate for a Brownian motion on the real free sphere. Lévy processes on the later quantum space were classified by B. Das, U. Franz and X. Wang in \cite[Thm 7.5]{das2018invariant} using a formula similar to Equation \eqref{eq:quantumhunt}, i.e. involving a positive constant $b$ and a Lévy measure $\nu$. Taking as before $b = 1$ and $\nu = 0$ yields a reasonable notion of a Brownian motion on $X_{N}$ which \emph{is not} the projection of the one on $O_{N}^{+}$. The convolution semigroup of states $(\varphi_{t})_{t\in \R_{+}}$ we are interested in is then given by :
\begin{equation*}
\varphi_{t} : u_{i1}^{n} \mapsto \delta_{i1}e^{-tR_{n}'(1)},
\end{equation*}
where the polynomials $(R_{n})_{n\in \N}$ are the orthogonal polynomials associated to the spectral measure of $u_{11}$ (see \cite{banica2009spectral} for details and explicit computations).

\begin{prop}
The $O_{N}^{+}$-invariant Lévy process on $X_{N}$ given by $(\varphi_{t})_{t\in \R_{+}}$ exhibits a cutoff phenomenon at time $t_{N} = \frac{1}{2}\ln(N)$.
\end{prop}

\begin{proof}
For $t$ large enough, $\varphi_{t}$ has an $L^{2}$-density with respect to $\omega = h\circ \pi$ given by
\begin{equation*}
\sum_{n=0}^{+\infty}d_{n}e^{-tR_{n}'(1)}u_{11}^{n}
\end{equation*}
so that
\begin{equation*}
\left\|\varphi_{t} - \omega\right\|^{2} \leqslant \frac{1}{4}\left\|\rho_{t} - 1\right\|_{1}^{2} \leqslant \frac{1}{4}\left\|\rho_{t} - 1\right\|_{2}^{2} = \frac{1}{4}\sum_{n=1}^{+\infty}d_{n}^{2}e^{-2tR_{n}'(1)}\|u_{11}^{n}\|_{2}^{2} = \frac{1}{4}\sum_{n=1}^{+\infty}d_{n}e^{-2tR_{n}'(1)}.
\end{equation*}
Now, we know from \cite[Cor 7.14]{das2018invariant} that
\begin{equation*}
n\leqslant R_{n}'(1)\leqslant \frac{N-1}{N-2}n.
\end{equation*}
and combining this with \cite[Lem 3.3]{freslon2017cutoff} shows that $t = \frac{1}{2}(\ln(N) + c)$ is enough to ensure the existence of the $L^{2}$-density and that
\begin{equation*}
\left\|\varphi_{t} - \omega\right\|^{2} \leqslant \frac{1}{4}\frac{1}{(1-q^{2})}\sum_{n=1}^{+\infty}q^{-n}e^{-2tn} = \frac{1}{4}\frac{1}{(1-q^{2})}\frac{1}{qe^{2t}-1} \leqslant \frac{1}{2}\frac{e^{-c}}{1 - e^{-c}}
\end{equation*}
and the upper bound follows.

The lower bound is proven as in Proposition \ref{prop:cutoff_sphere}.
\end{proof}

\begin{rem}
Note that there is an abuse of notations since the polynomials $R_{n}$ also depend on the integer $N$. This is different from the case $O_{N}^{+}$ where for all $N$, the orthogonal polynomials were always the same Chebyshev polynomials $P_{n}$. That fact, combined with the cumbersome available descriptions of $R_{n}$ (see for instance \cite[Sec 7.3]{das2018invariant}), make it difficult to compute the cutoff profile. However, because $\sqrt{N}u_{11}$ becomes semi-circular when $N$ goes to infinity by \cite[Thm 6.1]{banica2007integration}, it is reasonable to conjecture that $\sqrt{N}^{n}R_{n}$ converges to $P_{n}$ and the process has the same cutoff profile.
\end{rem}

\section{Quantum permutations}\label{sec:permutations}

Our second family of examples will be quantum permutations. The quantum permutation groups $S_{N}^{+}$ were introduced by Sh. Wang in \cite{wang1998quantum}. The corresponding $*$-algebra $\O(S_{N}^{+})$ is the quotient of $\O(O_{N}^{+})$ by the relations $u_{ij}^{2} = u_{ij}$. The coproduct factors through this and yields the compact quantum group structure. The connection to classical permutation may seem loose from that definition, but one easily shows that if $c_{ij} : S_{N}\to \C$ is the function sending a permutation $\sigma$ to $\pg \delta_{\sigma(i)j} \pd_{ij}$ , then there is a surjective $*$-homomorphism $\O(S_{N}^{+})\to \O(S_{N})$ sending $u_{ij}$ to $c_{ij}$, and that $\O(S_{N})$ is in fact the abelianization of $\O(S_{N}^{+})$. Thus, $S_{N}^{+}$ is a quantum version of $S_{N}$ somehow like $O_{N}^{+}$ is the quantum version of $O_{N}$. Beyond this fact which motivated the original definition, several connections between classical and quantum permutations have emerged which strongly support the idea that $S_{N}^{+}$ is the correct generalization of $S_{N}$. An example of particular interest from the probabilistic point of view is the free De Finetti theorem of C. Köstler and R. Speicher \cite{kostler2009noncommutative}.

The representation theory of $S_{N}^{+}$ is close to that of $O_{N}^{+}$, with the difference that when multiplying two characters (which are still indexed by the integers with $\chi_{0} = 1$ and $\chi_{0} + \chi_{1} = \summ{i=1}{N} u_{ii}$), one gets the formula
\begin{equation*}
\chi_{1}\chi_{n} = \chi_{n+1} + \chi_{n} + \chi_{n-1}.
\end{equation*}
The corresponding orthogonal polynomials are then given by the restriction to $[0, 4]$ of $Q_{n}(t) = P_{2n}(\sqrt{t})$, yielding the free Poisson law $\FPoiss(1,1)$ as spectral measure of $\chi_{1}$ under the Haar state. The associated dimensions of irreducible representations are $d_n = Q_n(N)$. 
We are now going to study two examples of processes on $S_{N}^{+}$, one continuous and one discrete.

\subsection{Brownian motion}
The natural candidate for Brownian motion on $S_{N}^{+}$ can be constructed exactly as in the case of $O_{N}^{+}$. Indeed, U. Franz, A. Kula and A. Skalski proved in \cite[Thm 10.10]{franz2016levy} a decomposition result for central Lévy processes on $S_{N}^{+}$ involving as before a positive constant $b$ and a Lévy measure $\nu$. Setting $b = 1$ and $\nu = 0$ leads to a central Lévy process. We will again denote by $(\lambda_n )_{n\in\N}$ the sequence determining the process, which is in this case given by
\begin{equation*}
\lambda_n  = \frac{Q_{n}'(N)}{Q_{n}(N)}.
\end{equation*}
The previous arguments carry on almost verbatim to yield the cutoff phenomenon and one can once again describe the cutoff profile as a distance between two free Meixner laws.

\begin{thm}\label{thm:cutofflevypermutation}
The central Lévy process defined above exhibits a cutoff phenomenon at time $N\ln(N)$. Moreover, setting again $t_c = N\ln(N) + cN$, for every $c\in\R$, we have
\begin{equation*}
\lim_{N\to +\infty}\left\|\psi_{t_c} - h\right\| = \left\|D_{\sqrt{1+e^{-c}}}\left(\FMeix\pg\frac{1-e^{-c}}{\sqrt{1+e^{-c}}}, \frac{-e^{-c}}{1+e^{-c}}\pd\right)\ast\delta_{e^{-c}} - \FMeix(0, 1)\right\| .
\end{equation*}
\end{thm}

\begin{proof}
As mentionned in Remark \ref{rem:profileimpliescutoff}, the existence of the cutoff profile implies the existence of the cutoff phenomenon, hence we will focus on the former. The proof proceeds as for Theorem \ref{thm:completeprofileorthogonal} and involves similar estimates. We will therefore focus on the features which differ.

Assuming first $c > 0$ and setting
\begin{equation*}
G_{c}(t) = \sum_{n=0}^{+\infty}e^{-cn}Q_{n}(t),
\end{equation*}
the cutoff profile equals
\begin{equation*}
\|G_{2c}(t) - 1\|_{1},
\end{equation*}
where the $L^1$ norm is computed with respect to the spectral measure of $\chi_{1}$ with respect to the Haar state, which is $\FPoiss(1,1)$. Note that because $P_{2n}$ is an even function and $P_{2n+1}$ an odd one,
\begin{equation*}
G_{2c}(t) = \sum_{n=0}^{+\infty}e^{-2cn}P_{2n}(\sqrt{t}) = \frac{F_{c}(\sqrt{t}) + F_{c}(-\sqrt{t})}{2}.
\end{equation*}
Setting $\beta = e^{-c}$ and $\gamma = \beta/(1+\beta^{2})$ as in Subsection \ref{subsec:limit_profile}, this leads to the formula
\begin{equation*}
G_{2c}(t) = \frac{1}{2(1+\beta^{2})}\left(\frac{1}{1-\gamma\sqrt{t}}+\frac{1}{1+\gamma\sqrt{t}}\right) = \frac{1}{1+\beta^{2}}\frac{1}{1-\gamma^{2}t}.
\end{equation*}
Let us also set
\begin{equation*}
\eta = \frac{1-\sqrt{1-4\gamma^{2}}}{2\gamma^{2}} = 1 + \beta^{2}.
\end{equation*}
Then, making the changes of variables $u = t-\eta$ and $v = u/\sqrt{\eta}$, and observing that $\gamma^{2} =(\eta - 1)\eta^{-2}$, the density of $G_{2c}(t)\mathrm{d}\FPoiss(1,1)(t)$ becomes
\begin{align*}
& \; \frac{1}{\eta}\frac{1}{1-\gamma^2t}\frac{1}{2\pi t}\sqrt{4-(t-2)^2} \mathbf{1}_{[0, 4]}(t)\mathrm{d}t \\
= & \; \frac{1}{2\pi\eta}\frac{1}{(1-\gamma^{2}(u+\eta))(u+\eta)}\sqrt{4 - (u - (2-\eta))^{2}}\mathbf{1}_{[-\eta, 4-\eta]}(u)\mathrm{d}u \\
= & \; \frac{1}{2\pi\eta}\frac{1}{(1-(\eta-1)\eta^{-2}(v\sqrt{\eta}+\eta))(v\sqrt{\eta} + \eta)}\sqrt{4 - (v\sqrt{\eta} - (2-\eta))^{2}}\mathbf{1}_{[-\sqrt{\eta}, \frac{4}{\sqrt{\eta}}-\eta]}(v)\sqrt{\eta}\mathrm{d}v \\
= &\;  \frac{1}{2\pi}\frac{1}{1+v(2-\eta)/\sqrt{\eta} + v^2(1-\eta)/\eta}\sqrt{4/\eta - (v - (2-\eta)/\sqrt{\eta})^{2}}\mathbf{1}_{[-\sqrt{\eta}, \frac{4}{\sqrt{\eta}}-\eta]}(v)\mathrm{d}v.
\end{align*}
 Setting $a = (2-\eta)/\sqrt{\eta}$, and $b = (1-\eta)/\eta$, this is exactly the density of the standardised free Meixner law with parameters $a$ and $b$,
 \begin{equation*}
 \frac{1}{2\pi}\frac{\sqrt{4(1+b) - (v-a)^{2}}}{1+av+bv^{2}}\mathbf{1}_{a-2\sqrt{1+b}, a+2\sqrt{1+b}}\mathrm{d}v.
\end{equation*}
Thus, $G_{2c}(t)\mathrm{d}\FPoiss(1, 1)(t)$ is the density of the law 
\begin{equation*}
D_{\sqrt{\eta}}\left(\FMeix\pg\frac{2-\eta}{\sqrt{\eta}}, \frac{1-\eta}{\eta}\pd\right)\ast\delta_{\eta}.
\end{equation*}
Writing $\FPoiss(1,1) = \FMeix(0, 1)\ast\delta_{1}$, applying $\ast\:\delta_{-1}$ on both sides and replacing $2c$ by $c$ now yields the desired result.

Assume now that $c < 0$. Let us first mention that the free Meixner distribution in the statement then has an atom given by the following formula :
\begin{equation*}
(1-e^{c})\delta_{e^{c}\sqrt{1+e^{-c}}}.
\end{equation*}
Applying the dilation by a factor $\sqrt{1+e^{-c}}$ and the translation by $e^{-c}$ changes the indices of the Dirac mass into $e^{c} + 1 + e^{-c}$ and translating again by $1$ to turn $\FMeix(0, 1)$ into $\FPoiss(1, 1)$, we see that the atom of the measures $m_{t_{c}}^{(N)}$ has to converge to $e^{c} + 2 + e^{-c} = (e^{c/2} + e^{-c/2})^{2}$.
The same argument as in Proposition \ref{thm:negativeprofile} yields an explicit formula for the measure of the corresponding classical process at time $t_{c}$ and the proof is done similarly. Note that by \cite{bozejko2006class}, the cumulants of $\FMeix(a, b)$ are polynomials in $a$ and $b$, hence in our case Laurent polynomials in $e^{-c}$, so that the analogue of Lemma \ref{lem:chebyshevmoments} holds.
\end{proof}

We can give an interpretation of this result similar to the one for Theorem \ref{thm:orthogonalcutoff}. Indeed, the function giving the number of fixed points of a permutation is, in terms of the generators of $\O(S_{N})$, $F = \sum c_{ii}$. Therefore, the elements $\chi_{1} = \sum u_{ii}$ is the quantum version of the number of fixed points. In particular, its law with respect to the Haar state, which is $\FPoiss(1,1)$, can be considered as the ``fixed points law for quantum permutations''. As a consequence, the difference between Brownian motion and the uniform measure on $S_{N}^{+}$ is asymptotically due to the fact that Brownian motion has ``too many fixed points''.

\subsection{Quantum random transpositions}

We will conclude with a discrete example, namely the quantum random transposition walk on the quantum permutation group. The reason for this is that the second-named author recently computed the cutoff profile for the classical version of that walk, while nothing is known in the quantum case.

Recall that if $\mu_{\mathrm{tr}}$ is the uniform measure on the set of transpositions, then the classical random transposition walk has increment distribution
\begin{equation*}
\mu = \frac{N-1}{N}\mu_{\mathrm{tr}} + \frac{1}{N}\delta_{e}.
\end{equation*}
One of the first results in the theory of cutoff phenomenon was the proof by P. Diaconis and M. Shahshahani in \cite{diaconis1981generating} that the random transposition walk exhibits a cutoff phenomenon at $\frac{1}{2}N\ln(N)$ steps. The second named author proved in \cite{teyssier2019limit} that the cutoff profile has the following form : for any $c\in \R$,
\begin{equation*}
\dvt\pg\mu^{\ast \frac{1}{2}(N\ln(N) + cN)},h\pd \xrightarrow[N\to\infty]{} \dvt\pg\mathrm{Poiss}\pg 1+e^{-c}\pd, \mathrm{Poiss}(1)\pd.
\end{equation*}
Note that $\mathrm{Poiss}(1)$ is the asymptotic law of the number of fixed points of a uniformly distributed permutation, which is the same as the law of the trace of a permutation matrix under the Haar measure, i.e. the law of $\chi_{1} + \chi_{0} = \chi_{1} + 1$.

The $\delta_{e}$-part in the definition of $\mu$ appears in a natural way. If we had decided to work with the \textit{pure} random transposition walk, working with the transition law $\mu_{\mathrm{tr}}$ instead of $\mu$, we would have had periodicity issues, as the signature would alternate from $1$ to $-1$. Including an extra $\delta_{e}$-part (sometimes refered to as the ``laziness'' of the walk) in the definition of $\mu$ is the way used by P. Diaconis and M. Shahshahani to rule out this problem. Note that in this case it is very natural to have a coefficient $1/N$ if one thinks of the random walk as a card shuffle : spread a deck of $N$ cards on a table and then choose two cards uniformly at random and swap them if they are different. The probability that the same card has been selected twice is then exactly $1/N$.

On the quantum side, there is a natural analogue of $\mu_{\mathrm{tr}}$ introduced in \cite{freslon2017cutoff} and denoted by $\varphi_{\mathrm{tr}}$. This is a central state given on the characters by
\begin{equation*}
\varphi_{\mathrm{tr}}(\chi_{n}) = Q_{n}(N-2),
\end{equation*}
where $Q_n(N) = P_{2n}(\sqrt{N})$ as in the previous subsection. We may then consider the quantum analogue of $\mu$, the walk given by
\begin{equation*}
\varphi = \frac{N-1}{N}\varphi_{\mathrm{tr}} + \frac{1}{N}\varepsilon.
\end{equation*}
It was proven in \cite{freslon2017cutoff} that the random walk on $S_{N}^{+}$ corresponding to $\varphi_{\mathrm{tr}}$ exhibits a cutoff phenomenon (with the same caveat as in Remark \ref{rem:truecutoff}), and that there is no periodicity issue, unlike for the classical walk. Consequently, in the quantum setting we can also naturally work with the pure random walk without extra laziness. As for the quantum lazy random walk associated with $\varphi$, the study of cutoff phenomenon becomes more delicate and was left as an open problem in \cite{freslon2017cutoff}. In this subsection we will show that the two random walks are asymptotically identical and hence have the same cutoff profile, which we will then compute.

For notational simplicity, we write for a state $\psi$ on $\O (S_N^+ )$
\begin{equation*}
\psi(n) = \frac{1}{d_{n}} \psi(\chi_{n}).    
\end{equation*}
It is well-known and easy to see from definition that $\psi^{*k} (n) = \psi(n)^{k}$ for all $k\in \N$. Lemma \ref{lem:upperbound} still holds if we replace $\psi_t$ by $\psi^{*k}$ and $\psi_t = e^{-t\lambda_n}$ by $\psi(n)^k$.

\subsubsection{The pure walk}

We start by revisiting the work of \cite{freslon2017cutoff} concerning the pure quantum transposition walk. There, the upper bound of cutoff phenomenon was shown to happen at time $N\ln(N)/2$, while the lower bound can be deduced from these computations and the methods used in the proof of Theorem \ref{thm:orthogonalcutoff}.

It is therefore natural to wonder whether the same strategy as in the orthogonal case can also yield the cutoff profile. The answer turns out to be yes, but requires a fine estimate on $\varphi_{\text{tr}}(n)$, in the same spirit as the affine approximation of $\lambda_{n}$ in Lemma \ref{lem:affineestimate}.

\begin{lem}\label{lem:exponentialestimate}
There exist $a_{N}, b_{N} < 0$ depending only on $N$ and a function $c_{N}$ of $n$ such that for all $n\geqslant 0$,
\begin{equation*}
\varphi_{\rm{tr}}(n) = e^{a_{N}n + b_{N} + c_{N}(n)}.
\end{equation*}
Moreover, $0\geqslant c_{N}(n)\geqslant -2q(\sqrt{N-2})^{4n+2}$.
\end{lem}

\begin{proof}
First, set for this proof $q = q (\sqrt{N})$ and $p = q(\sqrt{N-2})$, and recall that
\begin{equation*}
\varphi_{\rm{tr}}(n) = \frac{Q_{n}(N-2)}{Q_{n}(N)} = \frac{P_{2n}(\sqrt{N-2})}{P_{2n}(\sqrt{N})} = \frac{p^{-(2n+1)} - p^{2n+1}}{p^{-1}-p}\left(\frac{q^{-(2n+1)} - q^{2n+1}}{q^{-1}-q}\right)^{-1}.
\end{equation*}
Factoring, we obtain
\begin{equation*}
\varphi_{\rm{tr}}(n) = \frac{p^{-(2n+1)} - p^{2n+1}}{q^{-(2n+1)} - q^{2n+1}}\frac{q^{-1}-q}{p^{-1}-p} = \pg \frac{p}{q}\pd^{-(2n+1)} \frac{q^{-1}-q}{p^{-1}-p}\frac{1 - p^{4n+2}}{1 - q^{4n+2}} = \pg \frac{q}{p}\pd^{2n} \frac{1-q^2}{1-p^2}\frac{1 - p^{4n+2}}{1 - q^{4n+2}},
\end{equation*}
so that setting  $a_{N} = 2\ln(q/p)$, $b_{N} = \ln\pg \frac{1-q^2}{1-p^2}\pd$ and
$c_N(n) = \ln\pg\frac{1 - p^{4n+2}}{1 - q^{4n+2}}\pd$ yields the first part of the statement. As $p\geqslant q$, we have that $b_N \leqslant 0$ and $c_N(n) \leqslant 0$.

Let us now prove the lower bound on $c_N(n)$. Using that for $x < 1$, $\ln(1-x) \leqslant -x$ and for $0 \leqslant x \leqslant 1/2$, $\ln(1-x) \geqslant -2\ln(2)x$, we have
\begin{equation*}
c_N(n) = \ln\pg 1 - p^{4n+2} \pd - \ln\pg 1 - q^{4n+2} \pd \geqslant -2\ln(2)p^{4n+2} + q^{4n+2} \geq -2p^{4n+2}.
\end{equation*}
\end{proof}

We can now compute the measure of the classical process to prove the convergence of the complete profile, defined through the formula
\begin{equation*}
\int_{0}^{N}Q_{n}(x)\mathrm{d}m_{k}^{(N)}(x) = \varphi(n)^{k}Q_{n}(N).
\end{equation*}
Because of Lemma \ref{lem:exponentialestimate}, it is natural to look for an $\widetilde{N}(k)$ such that $Q_{n}(\widetilde{N}(k)) \approx \varphi_{\rm{tr}}^{*k}(n)Q_{n}(N)$, which leads heuristically to
\begin{equation*}
\widetilde{N}(k) = \left(e^{-ka_{N}/2}q + e^{k a_{N}/2}q^{-1}\right)^{2}
\end{equation*}

\begin{thm}\label{thm:profilepuretranspositions}
Set
\begin{equation*}
\alpha(k) = e^{-a_{N}k/2 + b_{N}k}\frac{e^{a_{N}k/2}q^{-1} - e^{-a_{N}k/2}q}{q^{-1} - q}.
\end{equation*}
Then, for any $c < 0$, and $N$ large enough, setting $k_{c} = \lceil (N\ln(N)+cN)/2\rceil$,
\begin{equation*}
m_{k_{c}}^{(N)} = \alpha(k_{c})\delta_{\widetilde{N}(k_{c})} + \sum_{n=0}^{+\infty}\left[\varphi_{\rm{tr}}^{\ast k_{c}}(n)Q_{n}(N) - Q_{n}(\widetilde{N}(k_{c}))\right]Q_{n}\mathrm{d}\FPoiss(1, 1) .
\end{equation*}
Moreover, for all $c\in \R$,
\begin{equation}\label{eq:profiltranspositions}
\left\|\varphi_{\rm{tr}}^{\ast k_{c}} - h\right\| \underset{N\to+\infty}{\longrightarrow} \left\|D_{\sqrt{1+e^{-c}}}\left(\FMeix\pg\frac{1-e^{-c}}{\sqrt{1+e^{-c}}}, \frac{-e^{-c}}{1+e^{-c}}\pd\right)\ast\delta_{e^{-c}} - \FMeix(0, 1)\right\|_{TV}.
\end{equation}
\end{thm}

\begin{proof}
We start by proving \eqref{eq:profiltranspositions} for $c > 0$. This follows from the same argument as in Proposition \ref{prop:rightprofileorthogonal} using the following computations : by Proposition \ref{prop:dlpn}, we have for every fixed $n\geq 1$, as $N\to\infty$,
\begin{equation*}
Q_n(N) = P_{2n}(\sqrt{N}) = N^n\pg 1 - \frac{1}{N} + O\pg\frac{1}{N^2}\pd\pd^{2n-1}
\end{equation*}
and
\begin{equation*}
\varphi_{\mathrm{tr}} (n) = \frac{Q_{n}(N-2)}{Q_{n}(N)} = \pg 1 - \frac{2}{N}\pd^n \pg \frac{1 - \frac{1}{N-2} + O\pg\frac{1}{N^2}\pd}{1 - \frac{1}{N} + O\pg\frac{1}{N^2}\pd}\pd^{2n-1} = \pg 1 - \frac{2}{N}\pd^n \pg 1  + O\pg\frac{1}{N^2}\pd\pd^{2n-1}
\end{equation*}
where we recall that $\varphi_{\mathrm{tr}} (n) = d_{n}^{-1}\varphi_{\mathrm{tr}}(\chi_{n})$. For $c\in \R$ fixed and $k_{c} = \lceil (N\ln(N)+cN)/2\rceil$, we have
\begin{align*}
\pg 1 - \frac{2}{N}\pd^{nk_{c}}
& = \exp\pg n \left\lceil\frac{1}{2}(N\ln(N) + cN)\right\rceil\pg -\frac{2}{N} + O\pg\frac{1}{N^2}\pd\pd\pd \\
& = \exp\pg n \pg -\ln(N) -c + O\pg\frac{\ln(N)}{N}\pd\pd\pd \\
& = \frac{e^{-n(c+o(1))}}{N^n}.
\end{align*}
It follows that
\begin{equation*}
d_{n}\varphi_{\text{tr}}(n)^{k_{c}} = N^{n} e^{no(1)} \frac{e^{-n(c+o(1))}}{N^{n}} = e^{-n(c+o(1))} \underset{N\to+\infty}{\longrightarrow} e^{-nc}.
\end{equation*}
Moreover, for $N$ large enough we have $c+o(1) > c/2 > 0$, hence we have a uniform bound $d_{n}\varphi_{\text{tr}}(n)^{k_{c}} \leqslant e^{-nc/2}$, which is summable with respect to $n$ and enables us to conclude when $c>0$.

We now assume, until the end of the proof, that $c < 0$. Let us set $q = q(\sqrt{N})$ and omit the $N$ indices for $a_{N}, b_{N}, c_{N}(n)$ for convenience. Then,
\begin{align*}
(q^{-1} - q)\alpha(k)Q_{n}(\widetilde{N}(k)) & = e^{-ak/2+bk}\left(e^{a(2n+1)k/2}q^{-(2n+1)} - e^{-a(2n+1)k/2}q^{2n+1}\right) \\
& = e^{ank + bk}q^{-(2n+1)} - e^{-a(n+1)k + bk}q^{2n+1} \\
& = (q^{-1} - q)e^{ank + bk}Q_{n}(N) + e^{ank + bk}q^{2n+1} - e^{-a(n+1)k + bk}q^{2n+1} \\
& = (q^{-1} - q)\varphi_{\rm{tr}}(n)^{k}e^{-c(n)k}Q_{n}(N) + e^{b k}\left(e^{ank} - e^{-a(n+1)k}\right)q^{2n+1},
\end{align*}
so that
\begin{equation*}
\abs{\varphi_{\text{tr}}^{\ast k}(n)Q_{n}(N) - \alpha(k) Q_{n}(\widetilde{N}(k))} \leqslant \left\vert 1 - e^{-c(n)k}\right\vert\varphi_{\rm{tr}}(n)^{k}Q_{n}(N) + \frac{e^{bk}q^{-1}}{q^{-1} - q}\left\vert e^{ank} - e^{-a(n+1)k}\right\vert q^{2n+2}.
\end{equation*} 
Let us now bound both terms for $k = k_{c}$ :
\begin{itemize}
\item For the first term we use the fact that $\varphi_{\rm{tr}}(n)^{k}Q_{n}(N) = O(e^{-cn})$ together with
\begin{equation*}
\left\vert 1 - e^{-c(n)k}\right\vert\leqslant 2\vert c(n)\vert k\leqslant 4kq\pg \sqrt{N-2} \pd^{4n+2}.
\end{equation*}  
\item As for the second term, the fraction involving $b$ in front of the absolute value does not depend on $n$, hence has no impact on the summability, while $a < 0$ for $N$ large enough so that we can bound the rest by
\begin{equation*}
e^{-a(n+1)k}q^{2n+2} = \left(e^{\ln(N) + c + o\left(\frac{\ln(N)}{N}\right)}q^{2}\right)^{n+1} = \left(e^{c + o\left(\frac{\ln(N)}{N}\right)}\left(1 + O\left(1/N\right)\right)\right)^{n+1}.
\end{equation*}
For $N$ large enough, the term in parenthesis is bounded by $e^{c/2}$, hence the whole thing is summable with respect to $n$.
\end{itemize}
The proof is now concluded as in the proof of Proposition \ref{thm:negativeprofile} to obtain the measure and the convergence of the profile for $c < 0$. Note that $\alpha(k)$ converges to $1-e^{c}$, which is the mass of the discrete part of the limit, and that $\widetilde{N}(k)$ converges to $\pg e^{c/2} + e^{-c/2} \pd^{2}$.

Remark that the case $c = 0$, which was not treated above, follows from the monotonicity of the distance proven in Lemma \ref{lem:decreasing} and the continuity of the profile at $0$.
\end{proof}

\subsubsection{The lazy walk}
We now turn to the lazy random walk associated with the state
\begin{equation*}
\varphi = \frac{N-1}{N}\varphi_{\mathrm{tr}} + \frac{1}{N}\varepsilon.
\end{equation*}
As mentioned previously, the study of the corresponding cutoff phenomenon is subtler.
Indeed, as pointed out in \cite{freslon2017cutoff}, the states $\varphi^{*k}$ never admit $L^{2}$-densities and hence the previous method based on Lemma \ref{lem:upperbound} for $c > 0$ does not work any more. Our idea in the present work will be to approximate the lazy random walk by the pure one by mimicking an alternative classical way of avoiding periodicity for Markov chains, which was used for example to study the cutoff for $k$-cycles by Berestycki, Schramm and Zeitouni \cite{berestycki2011mixing}. It consists in working in continuous time and considering a clock which rings at a random time given by an exponential law of parameter one. Each time the clock rings, we make one step, and reset the clock. Note that the standard deviation of a sum of order $N\ln(N)$ independent variables of law $\mathrm{Exp}(1)$ is of the order $\sqrt{N\ln(N)}$, which is negligible when compared to $N$, the size of the cutoff window, so it doesn't change the cutoff profile at all. A similar comment can be made about adding extra laziness as long as the laziness coefficient is not too large.

As the next result will show, one can transfer this idea to the quantum setting and the result is formally equivalent to adding laziness. Moreover, it leads to a simple proof of the cutoff phenomenon. 

\begin{thm}\label{cor:profileparesse}
The random walk associated to $\varphi=\frac{N-1}{N}\varphi_{\mathrm{tr}} + \frac{1}{N}\varepsilon$ exhibits a cutoff phenomenon at $N\ln(N)/2$ steps in Fourier-Stieltjes norm. Moreover, the associated cutoff profile is given, for every $c\in\R$, by
\begin{equation*}
\lim_{N\to +\infty}\left\|\varphi ^{\ast\left\lceil\frac{1}{2}(N\ln(N) + cN)\right\rceil} - h\right\| = \left\|D_{\sqrt{1+e^{-c}}}\left(\FMeix\pg\frac{1-e^{-c}}{\sqrt{1+e^{-c}}}, \frac{-e^{-c}}{1+e^{-c}}\pd\right)\ast\delta_{e^{-c}} - \FMeix(0, 1)\right\|_{TV}.
\end{equation*}
\end{thm}

\begin{proof}
As explained in Remark \ref{rem:profileimpliescutoff}, it suffices to prove the cutoff profile. For each $j\in\N$, we denote by $X_{j} $ a random variable following the binomial distribution of parameters $j$ and $p = 1-1/N$. Then for any $j\in\N$, we have
\begin{equation*}
\mathbb{E}\pg \varphi_{\mathrm{tr}}^{\ast X_j} \pd = \sum_{k=0}^j \binom{j}{k} \left(\frac{N-1}{N}\right)^k \left(\frac{1}{N}\right)^{n-k} \varphi_{\mathrm{tr}}^{\ast k} = \varphi^{\ast j},
\end{equation*}
where we make the convention that $\varphi_{\mathrm{tr}}^{\ast 0} = \varepsilon$. Write $k_{c} = \left\lceil\frac{1}{2}(N\ln(N) + cN)\right\rceil$. 
We have
$$\left\| \mathbb{E}\pg \varphi_{\mathrm{tr}}^{\ast X_{k_{c}}}\pd  - \varphi_{\mathrm{tr}}^{\ast k_{c}} \right\|_{FS} 
=\left\| \mathbb{E}\left(\varphi_{\mathrm{tr}}^{\ast X_{k_{c}}}  - \varphi_{\mathrm{tr}}^{\ast k_{c}} \right) \right\|_{FS}
\leqslant \mathbb{E} \left\| \varphi_{\mathrm{tr}}^{\ast X_{k_{c}}}  - \varphi_{\mathrm{tr}}^{\ast k_{c}} \right\|_{FS} .
$$
Note that $\mathbb E \pg X_{k_{c}} \pd = pk_{c} = k_{c} + O(\ln (N))$, so that by the Markov inequality,
\begin{equation*}
\mathbb{P}\pg\abs{X_{k_{c}} - k_{c}} < \sqrt{N}\pd \xrightarrow[N\to\infty]{} 1.
\end{equation*}
Using this together with the dominated convergence theorem yields
\begin{align*}
\lim_{N\to +\infty} \left\| \mathbb{E}\pg\varphi_{\mathrm{tr}}^{\ast X_{k_{c}}} \pd  - \varphi_{\mathrm{tr}}^{\ast k_{c}} \right\|_{FS} &
\leqslant 
\lim_{N\to +\infty} \mathbb{E}\left( \mathbf 1 _{\abs{X_{k_{c}} - k_{c}} < \sqrt{N}} \left\| \varphi_{\mathrm{tr}}^{\ast X_{k_{c}} }  - \varphi_{\mathrm{tr}}^{\ast k_{c} } \right\|_{FS} \right) \\
& = 
\mathbb{E}\left( \lim_{N\to +\infty}  \mathbf 1 _{\abs{X_{k_{c}} - k_{c}} < \sqrt{N}} \left\| \varphi_{\mathrm{tr}}^{\ast X_{k_{c}} }  - \varphi_{\mathrm{tr}}^{\ast k_{c} } \right\|_{FS} \right).
\end{align*}

Let us first assume $c > 0$. The proof of Theorem \ref{thm:profilepuretranspositions} shows that for any sequence $k_{N} = k_{c} + o(N)$, we have
\begin{equation*}
d_{n} \varphi_{\mathrm{tr}} (n)^{k_{N}} = e^{-n(c+o (1) )}
\end{equation*}
as $N\to+\infty$, and moreover $d_{n}\varphi_{\text{tr}}(n)^{k_{N}}\xrightarrow[N\to\infty]{} e^{-nc}$ and $d_{n}\varphi_{\text{tr}}(n)^{k_{N}} \leqslant e^{-nc/2}$ for $N$ large enough. In particular, applying the same strategy of exchange of sum and limit as in the proof of Proposition \ref{thm:negativeprofile}, we have
\begin{equation*}
\left\| \varphi_{\mathrm{tr}}^{\ast k_{N}}  - \varphi_{\mathrm{tr}}^{\ast k_{c}} \right\|_{FS} \xrightarrow[N\to\infty]{} 0.
\end{equation*}
Together with the previous estimates, this yields 
\begin{equation*}
\lim_{N\to +\infty} \left\| \varphi^{\ast k_{c}}  - \varphi_{\mathrm{tr}}^{\ast k_{c}} \right\|_{FS} \leqslant
\mathbb{E}\left( \lim_{N\to +\infty}  \mathbf 1 _{\abs{X_{k_{c}} - k_{c}} < \sqrt{N}} \left\| \varphi_{\mathrm{tr}}^{\ast X_{k_{c}} }  - \varphi_{\mathrm{tr}}^{\ast k_{c} } \right\|_{FS} \right) =0 .
\end{equation*}
and the result follows from Theorem \ref{thm:profilepuretranspositions}.

Now assume $c < 0$ and consider the measures $m_k^{(N)}$ associated with the pure random walk which were defined just before Theorem \ref{thm:profilepuretranspositions}. Note that for any $k\in\N$, $ \mathbb{E}\pg m_{X_k}^{(N)} \pd $ is a bounded measure  on $[0,N]$ such that
\begin{equation*}
\left\|\varphi^{\ast k} - h \right\| = \left\| \mathbb{E}\pg m_{X_k}^{(N)} \pd - \FPoiss(1, 1) \right\|_{TV}.
\end{equation*}
Note also that
\begin{equation*}
  \left\|\mathbf 1 _{[0,4]}  \left(  \mathbb{E}\pg m_{X_{k_c}}^{(N)} \pd - m_{k_c} ^{(N)}  \right) \right\|_{TV} 
 = \left\|  \mathbb{E}\left( \mathbf 1 _{[0,4]} \pg m_{X_{k_c}}^{(N)}  - m_{k_c} ^{(N)} \pd \right) \right\|_{TV}  \leqslant   \mathbb{E}\left\|    \mathbf 1 _{[0,4]} \pg m_{X_{k_c}}^{(N)}  - m_{k_c} ^{(N)} \pd  \right\|_{TV}
\end{equation*}
The same argument as for the case $c>0$ yields that the right hand side tends to $0$. So together with Theorem \ref{thm:profilepuretranspositions}, we deduce that 
$$\lim_{N\to+\infty}\left\|\mathbf 1 _{[0,4]} \pg\mathbb{E}\pg m_{X_{k_c}}^{(N)} \pd - \FPoiss(1, 1) \pd \right\|_{TV} = 
\lim_{N\to+\infty}\left\|\mathbf 1 _{[0,4]} \pg  m_{ k_c}^{(N)}  - \FPoiss(1, 1) \pd \right\|_{TV}  . $$
On the other hand, using again Theorem \ref{thm:profilepuretranspositions}, we see that
\begin{align*}
  \left\|\mathbf 1 _{\R \setminus [0,4]}  \pg\mathbb{E}\pg m_{X_{k_c}}^{(N)} \pd - \FPoiss(1, 1) \pd \right\|_{TV} 
 & =  \left\|\mathbf 1 _{\R \setminus [0,4]}  \mathbb{E}\pg m_{X_{k_c}}^{(N)} \pd  \right\|_{TV} 
= \left\| \mathbb{E}\left( \alpha(X_{k_{c}})\delta_{\widetilde{N}(X_{k_{c}})} \right) \right\|_{TV}\\
& = \left\| \sum_{i\in \N} \mathbb{P}  ( X_{k_{c}} = i)\alpha(i) \delta_{\widetilde{N}(i)} \right\|_{TV} 
= \sum_{i\in \N} \mathbb{P}  ( X_{k_{c}} = i)\alpha(i)\\
& = \mathbb{E}\left( \alpha(X_{k_{c}}) \right)  ,
\end{align*}
which tends to the mass of the discrete part of the free Meixner law in the desired profile. The proof is now concluded as in the proof of Theorem \ref{thm:profilepuretranspositions}.
\end{proof}

\bibliographystyle{amsplain}
\bibliography{Cut-off}

\providecommand{\bysame}{\leavevmode\hbox to3em{\hrulefill}\thinspace}
\providecommand{\MR}{\relax\ifhmode\unskip\space\fi MR }
\providecommand{\MRhref}[2]{%
  \href{http://www.ams.org/mathscinet-getitem?mr=#1}{#2}
}
\providecommand{\href}[2]{#2}
\begin{thebibliography}{10}

\bibitem{banica1996theorie}
T.~Banica, \emph{{Th{\'e}orie des repr{\'e}sentations du groupe quantique
  compact libre $O(n)$}}, C. R. Acad. Sci. Paris S\'{e}r. I Math. \textbf{322}
  (1996), no.~3, 241--244.

\bibitem{banica1997groupe}
\bysame, \emph{{Le groupe quantique compact libre $U(n)$}}, Comm. Math. Phys.
  \textbf{190} (1997), no.~1, 143--172.

\bibitem{banica2007integration}
T.~Banica and B.~Collins, \emph{{Integration over quantum permutation groups}},
  J. Funct. Anal. \textbf{242} (2007), no.~2, 641--657.

\bibitem{banica2009spectral}
T.~Banica, B.~Collins, and P.~Zinn-Justin, \emph{{Spectral analysis of the free
  orthogonal matrix}}, Int. Math. Res. Not. \textbf{2009} (2009), no.~17,
  3286--3309.

\bibitem{BayerDiaconis1992}
D.~Bayer and P.~Diaconis, \emph{Trailing the dovetail shuffle to its lair},
  Ann. Appl. Probab. \textbf{2} (1992), no.~2, 294--313.

\bibitem{bercovici1999stable}
H.~Bercovici and V.~Pata, \emph{{Stable laws and domains of attraction in free
  probability theory. With an appendix by P. Biane}}, Ann. Math. (1999),
  1023--1060.

\bibitem{berestycki2011mixing}
{Berestycki, N. and Schramm, O. and Zeitouni, O.}, \emph{{Mixing times for
  random k-cycles and coalescence-fragmentation chains}}, Ann. Probab.
  \textbf{39} (2011), no.~5, 1815--1843.

\bibitem{biane2008introduction}
Philippe Biane, \emph{Introduction to random walks on noncommutative spaces},
  Quantum potential theory, Springer, 2008, pp.~61--116.

\bibitem{blackadar2006operator}
B.~Blackadar, \emph{{Operator algebras}}, Encyclop\ae{}dia of Mathematical
  Sciences, vol. 122, Springer, 2006.

\bibitem{bozejko2006class}
M.~Bo{\.z}ejko and W.~Bryc, \emph{{On a class of free L{\'e}vy laws related to
  a regression problem}}, J. Funct. Anal. \textbf{236} (2006), no.~1, 59--77.

\bibitem{brannan2011approximation}
M.~Brannan, \emph{{Approximation properties for free orthogonal and free
  unitary quantum groups}}, J. Reine Angew. Math. \textbf{672} (2012),
  223--251.

\bibitem{brannanruan2017lp}
M.~Brannan and Z-J. Ruan, \emph{{$L_p$}-representations of discrete quantum
  groups}, J. Reine Angew. Math. \textbf{732} (2017), 165--210.

\bibitem{chen2008cutoff}
G.~Chen and L.~Saloff-Coste, \emph{{The cutoff phenomenon for ergodic Markov
  processes}}, Electronic J. Probab. \textbf{13} (2008), 26--78.

\bibitem{cipriani2012symmetries}
F.~Cipriani, U.~Franz, and A.~Kula, \emph{{Symmetries of L\'evy processes,
  their Markov semigroups and potential theory on compact quantum groups}}, J.
  Funct. Anal. \textbf{266} (2014), no.~5, 2789--2844.

\bibitem{das2018invariant}
B.~Das, U.~Franz, and X.~Wang, \emph{{Invariant Markov semigroups on quantum
  homogeneous spaces}}, J. Noncommut. Geom. (2020).

\bibitem{diaconis1988group}
P.~Diaconis, \emph{{Group representations in probability and statistics}},
  Lecture Notes-Monograph Series, vol.~11, Institute of Mathematical
  Statistics, 1988.

\bibitem{diaconis1981generating}
P.~Diaconis and M.~Shahshahani, \emph{{Generating a random permutation with
  random transpositions}}, Prob. Theory Related Fields \textbf{57} (1981),
  no.~2, 159--179.

\bibitem{franz2006levy}
U.~Franz, \emph{L{\'e}vy processes on quantum groups and dual groups}, Quantum
  independent increment processes II (M.~Schurmann and U.~Franz, eds.), Lecture
  Notes in Mathematics, vol. 1866, Springer, 2006, pp.~161--257.

\bibitem{franz2017hypercontractivity}
U.~Franz, G.~Hong, F.~Lemeux, M.~Ullrich, and H.~Zhang,
  \emph{{Hypercontractivity of heat semigroups on free quantum groups}}, J.
  Operator Theory \textbf{77} (2017), no.~1, 61--76.

\bibitem{franz2016levy}
U.~Franz, A.~Kula, and A.~Skalski, \emph{{L{\'e}vy processes on quantum
  permutation groups}}, Noncommutative analysis, operator theory and
  applications, Birkhäuser, 2016, pp.~193--259.

\bibitem{freslon2018quantum}
A.~Freslon, \emph{{Quantum reflections, random walks and cut-off}}, Internat.
  J. Math. \textbf{29} (2018), no.~14, 1850101.

\bibitem{freslon2017cutoff}
\bysame, \emph{{Cut-off phenomenon for random walks on free orthogonal quantum
  groups}}, Probab. Theory Related Fields \textbf{174} (2019), no.~3--4,
  731--760.

\bibitem{kostler2009noncommutative}
C.~K{\"o}stler and R.~Speicher, \emph{{A noncommutative de Finetti theorem :
  invariance under quantum permutations is equivalent to freeness with
  amalgamation}}, Comm. Math. Phys. \textbf{291} (2009), no.~2, 473--490.

\bibitem{Lacoin2016}
H.~Lacoin, \emph{Mixing time and cutoff for the adjacent transposition shuffle
  and the simple exclusion}, Ann. Probab. \textbf{44} (2016), no.~2,
  1426--1487.

\bibitem{liao2004levy}
M.~Liao, \emph{{L{\'e}vy processes and Fourier analysis on compact Lie
  groups}}, Ann. Probab. (2004), 1553--1573.

\bibitem{mccarthy2018diaconis}
J.P. McCarthy, \emph{{Diaconis-Shahshahani upper bound lemma for finite quantum
  groups}}, J. Fourier Anal. App. \textbf{25} (2019), 2463–--2491.

\bibitem{meliot2014cut}
P.-L. M{\'e}liot, \emph{{The cut-off phenomenon for Brownian motions on compact
  symmetric spaces}}, Potential Anal. \textbf{40} (2014), no.~4, 427--509.

\bibitem{neshveyev2014compact}
S.~Neshveyev and L.~Tuset, \emph{{Compact quantum groups and their
  representation categories}}, Cours Sp\'ecialis\'es, vol.~20, Société
  Mathématique de France, 2013.

\bibitem{NestoridiThomas}
E.~Nestoridi and S.~Thomas, \emph{Limit profiles for markov chains}, arXiv
  preprint arXiv:2005.13437 (2020).

\bibitem{nica2006lectures}
A.~Nica and R.~Speicher, \emph{{Lectures on the combinatorics of free
  probability}}, Lecture Note Series, vol. 335, London Mathematical Society,
  2006.

\bibitem{coursSalez}
J.~Salez, \emph{{Temps de Mélange des Chaînes de Markov}}, Online Lecture
  Notes available at \url{www.ceremade.dauphine.fr/~salez/mixing.pdf} (2018).

\bibitem{takesaki2002theory}
M.~Takesaki, \emph{{Theory of operator algebras I}}, {Encyclop\ae{}dia of
  Mathematical Sciences}, vol. 124, Springer, 2002.

\bibitem{takesaki2003theoryII}
\bysame, \emph{{Theory of operator algebras II}}, {Encyclop\ae{}dia of
  Mathematical Sciences}, vol. 125, Springer, 2003.

\bibitem{teyssier2019limit}
L.~Teyssier, \emph{{Limit profile for random transpositions}}, Ann. Probab.
  \textbf{48} (2019), no.~5, 2323--2343.

\bibitem{timmermann2008invitation}
T.~Timmermann, \emph{{An invitation to quantum groups and duality. From Hopf
  algebras to multiplicative unitaries and beyond}}, EMS Textbooks in
  Mathematics, European Mathematical Society, 2008.

\bibitem{Voit1996}
M.~Voit, \emph{{Asymptotic distributions for the Ehrenfest urn and related
  random walks}}, J. Appl. Probab. (1996), 340--356.

\bibitem{wang1995free}
Sh. Wang, \emph{Free products of compact quantum groups}, Comm. Math. Phys.
  \textbf{167} (1995), no.~3, 671--692.

\bibitem{wang1998quantum}
\bysame, \emph{Quantum symmetry groups of finite spaces}, Comm. Math. Phys.
  \textbf{195} (1998), no.~1, 195--211.

\bibitem{woronowicz1995compact}
S.L. Woronowicz, \emph{{Compact quantum groups}}, Sym{\'e}tries quantiques (Les
  Houches, 1995) (1998), 845--884.

\end{thebibliography}

\end{document}